\documentclass[12pt,a4paper]{article}

\title{Quasi-Stationary Asymptotics for Perturbed Semi-Markov Processes in Discrete Time}
\author{Mikael Petersson\footnote{Department of Mathematics, Stockholm University, SE-106 91 Stockholm, Sweden, mikpe@math.su.se.}}
\date{}

\usepackage{amsfonts}
\usepackage{amsmath}
\usepackage{amsthm}

\newcommand{\eps}{{(\varepsilon)}}
\newcommand{\zero}{{(0)}}
\newcommand{\cond}{\, | \,}
\newcommand{\jA}{{_{j}\mathbf{A}}}
\newcommand{\iP}{{_{i}\mathbf{P}}}
\newcommand{\jP}{{_{j}\mathbf{P}}}
\newcommand{\jU}{{_{j}\mathbf{U}}}
\newcommand{\ja}{{_{j}a}}
\newcommand{\ig}{{_{i}g}}
\newcommand{\jg}{{_{j}g}}
\newcommand{\kg}{{_{k}g}}
\newcommand{\jp}{{_{j}p}}
\newcommand{\iphi}{{_{i}\phi}}
\newcommand{\jphi}{{_{j}\phi}}
\newcommand{\kphi}{{_{k}\phi}}

\newcommand{\Ccontinuity}{\mathbf{A}}
\newcommand{\Cergodicity}{\mathbf{B}}
\newcommand{\Cmgf}{\mathbf{C}}
\newcommand{\Cperturbation}{\mathbf{D}}
\newcommand{\Cnonperiodicity}{\mathbf{E}}

\newcommand{\Ccontinuityreg}{\mathbf{A^*}}
\newcommand{\Cmgfreg}{\mathbf{B^*}}
\newcommand{\Cperturbationreg}{\mathbf{C^*}}
\newcommand{\Ctailprobreg}{\mathbf{D^*}}

\newcommand{\Cperturbationmom}{\mathbf{D'}}

\newtheorem{theorem}{Theorem}
\newtheorem{lemma}{Lemma}

\numberwithin{equation}{section}
\numberwithin{theorem}{section}
\numberwithin{lemma}{section}

\begin{document}

\maketitle

\begin{abstract}
We consider a discrete time semi-Markov process where the characteristics defining the process depend on a small perturbation parameter. It is assumed that the state space consists of one finite communicating class of states and, in addition, one absorbing state. Our main object of interest is the asymptotic behaviour of the joint probabilities of the position of the semi-Markov process and the event of non-absorption as time tends to infinity and the perturbation parameter tends to zero. The main result gives exponential expansions of these probabilities together with an recursive algorithm for computing the coefficients in the expansions.
\end{abstract}

\noindent \textbf{Keywords:} Semi-Markov process, Perturbation, Asymptotic Expansion, Regenerative process, Renewal equation, Solidarity property, First hitting time. \\
\\
\textbf{MSC2010:} Primary 60K15; Secondary 41A60, 60K05.

\section{Introduction}

The aim of this paper is to present a detailed asymptotic analysis of the long time behaviour of non-linearly perturbed discrete time semi-Markov processes with absorption.

We consider a discrete time semi-Markov process $\xi^\eps(n)$, on a finite state space, depending on a small perturbation parameter $\varepsilon \geq 0$ in the sense that its transition probabilities are functions of $\varepsilon$. It is assumed that these functions are continuous at $\varepsilon = 0$ so that the process $\xi^\eps(n)$ for $\varepsilon > 0$ can be interpreted as a perturbation of the process $\xi^\zero(n)$. Furthermore, we assume that for $\varepsilon$ small enough, the state space can be partitioned into one communicating class of states $\{ 1,\ldots,N \}$ and one absorbing state $0$. The absorption time, that is, the first hitting time of state $0$ for the semi-Markov process $\xi^\eps(n)$, is denoted by $\mu_0^\eps$.

Our main object of interest is the asymptotic behaviour of the probabilities
\begin{equation*}
 P_{i j}^\eps(n) = \mathsf{P}_i \{ \xi^\eps(n) = j, \ \mu_0^\eps > n \}, \ i,j \neq 0,
\end{equation*}
as $n \rightarrow \infty$ and $\varepsilon \rightarrow 0$.

It turns out that the forms of the asymptotic results depend on if one-step absorption probabilities vanish asymptotically or if some of them are non-zero in the limit. In the former case the absorption time $\mu_0^\eps \rightarrow \infty$ in probability as $\varepsilon \rightarrow 0$ and we get so-called \emph{pseudo-stationary} asymptotics for the probabilities $P_{i j}^\eps(n)$. In the latter case, $\mu_0^\eps$ are stochastically bounded as $\varepsilon \rightarrow 0$ and we get so-called \emph{quasi-stationary} asymptotics for the probabilities $P_{i j}^\eps(n)$. In the present paper we give a unified treatment of both cases.

Our perturbation conditions are formulated in terms of the following mixed power-exponential moments for transition probabilities:
\begin{equation} \label{eq:momentsintro}
 p_{i j}^\eps(\rho^\zero,r) = \sum_{n=0}^\infty n^r e^{\rho^\zero n} Q_{i j}^\eps(n), \ r=0,1,\ldots, \ i,j \neq 0,
\end{equation}
where $Q_{i j}^\eps(n)$ are the transition probabilities for the semi-Markov process and $\rho^\zero$ is a non-negative constant determined by the distribution of first return time to the initial state for the limiting semi-Markov process. In the pseudo-stationary case $\rho^\zero = 0$ and then the moments in \eqref{eq:momentsintro} reduce to usual power moments.

We allow for smooth non-linear perturbations which means that the moments in \eqref{eq:momentsintro} may be non-linear function of $\varepsilon$ which for $r = 0,\ldots,k$ can be expanded in an asymptotic power series with respect to $\varepsilon$.

As it turns out, the asymptotics of the probabilities $P_{i j}^\eps(n)$ depends on the balance between the rate at which the time $n \rightarrow \infty$ and the perturbation $\varepsilon \rightarrow 0$. If we write $n = n^\eps$ as a function of $\varepsilon$, this balance is characterized by the following relation:
\begin{equation} \label{eq:balanceintro}
 \varepsilon^r n^\eps \rightarrow \lambda_r \in [0,\infty), \ \text{for some} \ 1 \leq r \leq k.
\end{equation}

Under assumptions mentioned above and some additional Cram{\'e}r type conditions on moments of transition times and a non-periodicity condition for the limiting semi-Markov process we obtain the following which is our main result: For any $n^\eps \rightarrow \infty$ as $\varepsilon \rightarrow 0$ in such a way that relation \eqref{eq:balanceintro} holds, we have
\begin{equation} \label{eq:mainintro}
 \frac{ \mathsf{P}_i \{ \xi^\eps(n^\eps) = j, \ \mu_0^\eps > n^\eps \} }{\exp( -(\rho^\zero + c_1 \varepsilon + \cdots + c_{r-1} \varepsilon^{r-1}) n^\eps )} \rightarrow \frac{\widetilde{\pi}_{i j}^\zero}{e^{\lambda_r c_r}} \ \text{as} \ \varepsilon \rightarrow 0, \ i,j \neq 0.
\end{equation}
Relation \eqref{eq:mainintro} is supplemented with (i) an explicit expression for the constant $\widetilde{\pi}_{i j}^\zero$, (ii) an equation from which $\rho^\zero$ can be found at least numerically, and (iii) a recursive algorithm for computing the coefficients $c_1,\ldots,c_r$ as rational functions of coefficients in expansions of the moments in Equation \eqref{eq:momentsintro}.

In the pseudo-stationary case, the asymptotic relation \eqref{eq:mainintro} takes a simpler form. In this case, $\rho^\zero = 0$ and the constants $\widetilde{\pi}_{i j}^\zero$ do not depend on the initial state $i$ and are given by the stationary probabilities of the limiting semi-Markov process.

In order to prove \eqref{eq:mainintro} we use the theory of perturbed discrete time renewal equations developed in Gyllenberg and Silvestrov (1994), Englund and Silvestrov (1997), and Silvestrov and Petersson (2013). However, the results can not be applied directly. This is because conditions for semi-Markov processes are naturally formulated in terms of its transition probabilities while the application of the renewal theory requires conditions for some non-local characteristics of the semi-Markov process to hold. To prove that the conditions we formulate for semi-Markov processes are sufficient for the conditions required for the results from renewal theory we use techniques from Gyllenberg and Silvestrov (2008). In particular, we need to calculate the coefficients in expansions of mixed power-exponential moments for first return times based on the coefficients in the expansions of the moments in Equation \eqref{eq:momentsintro}. This analysis makes up a substantial part of the proof of the main result and may also have applications beyond the scope of this paper.

The asymptotic relation \eqref{eq:mainintro} is proved for continuous time semi-Markov processes in Gyllenberg and Silvestrov (1999, 2008). In Gyllenberg and Silvestrov (2008) the result is also extended to the case of initial 
transient states.

Expansions of the type given in Equation \eqref{eq:mainintro} and similar types of exponential expansions have also been given for ruin probabilities in perturbed risk models, see for example Gyllenberg and Silvestrov (2000, 2008), Englund (2001), Blanchet and Zwart (2010), Ni (2011, 2014), and Petersson (2014).

In the pseudo-stationary case, many of the existing results in the literature are concerned with an asymptotic analysis of absorption times or other types of first hitting times in various types of Markov and semi-Markov processes, see for example Keilson (1966), Latouche and Louchard (1978), Latouche (1991), Avrachenkov and Haviv (2004), Drozdenko (2007), and Jung (2013).

In the quasi-stationary case, almost all papers in the literature deals with models without perturbations. In particular, a great deal of attention has been given the study of so-called quasi-stationary distributions, see for example Darroch and Seneta (1965), Seneta and Vere-Jones (1966), Cheong (1970), Flaspohler and Holmes (1972), Collet, Mart{\'i}nez, and San Mart{\'i}n (2013), and van Doorn and Pollett (2013). For models with perturbations, asymptotic expansions of quasi-stationary distributions are given in Gyllenberg and Silvestrov (2008) for continuous time regenerative processes and semi-Markov processes, and in Petersson (2013) for discrete time regenerative processes.

One of the most extensively studied models of perturbed stochastic processes is the model of linearly perturbed Markov chains. In particular, asymptotic expansions of stationary distributions have been given for so-called nearly uncoupled Markov chains. For some results and more references related to this line of research we refer to Simon and Ando (1961), Schweitzer (1968), Stewart (1991), Hassin and Haviv (1992), Yin and Zhang (1998, 2003), Altman, Avrachenkov and N{\'u}{\~n}ez-Queija (2004), and Avrachenkov, Filar, and Howlett (2013).

For more references related to pseudo-stationary and quasi-stationary asymptotics we refer to the extensive bibliography given in Gyllenberg and Silvestrov (2008).

Let us finally briefly outline the structure of the paper. Section \ref{sec:regenerative} presents exponential expansions for perturbed discrete time regenerative processes. In Section \ref{sec:semimarkov} we present in detail the model of perturbed discrete time semi-Markov processes and introduce some notation that will be used throughout the paper. Section \ref{sec:moments} derives systems of linear equations for moments of first hitting times and gives a necessary and sufficient condition for these moments to be finite. In Section \ref{sec:solidarity} we prove some solidarity properties for moments of first hitting times which are essential for our main result. Section \ref{sec:powerseries} constructs asymptotic power series expansions for moments of first hitting times. In Section \ref{sec:periodicity} we prove a solidarity property of periodicity which is needed in order to apply the renewal theory. Finally, Section \ref{sec:mainresult} presents the main asymptotic result.

\section{Exponential Expansions for Perturbed \\ Regenerative Processes} \label{sec:regenerative}

This section presents asymptotic exponential expansions for perturbed discrete time regenerative processes. The results in this section are obtained by applying a corresponding result for discrete time renewal equations given in Silvestrov and Petersson (2013).

For every $\varepsilon \geq 0$, let $Z_n^\eps$, $n = 0,1,\ldots,$ be a regenerative process on a measurable state space $(\mathcal{X}, \Gamma)$ with proper regeneration times $0 = \tau_0^\eps < \tau_1^\eps < \cdots$. Furthermore, let $\mu^\eps$ be a random variable, defined on the same probability space, that takes values in the set $\{ 0,1,\ldots,\infty \}$. Assume that for each $A \in \Gamma$, the probabilities $P^\eps(n,A) = \mathsf{P} \{ Z_n^\eps \in A, \ \mu^\eps > n \}$ satisfy the renewal equation
\begin{equation*}
 P^\eps(n,A) = q^\eps(n,A) + \sum_{k=0}^n P^\eps(n-k,A) f^\eps(k), \ n = 0,1,\ldots,
\end{equation*}
where
\begin{equation*}
 q^\eps(n,A) = \mathsf{P} \{ Z_n^\eps \in A, \ \mu^\eps \wedge \tau_1^\eps > n \}
\end{equation*}
and
\begin{equation*}
 f^\eps(k) = \mathsf{P} \{ \tau_1^\eps = k, \ \mu^\eps > \tau_1^\eps \}.
\end{equation*}
Then, we call $\mu^\eps$ a regenerative stopping time.

Notice that $f^\eps(n)$ are possibly improper distributions with defect
\begin{equation*}
 f^\eps = 1 - \sum_{n=0}^\infty f^\eps(n) = \mathsf{P} \{ \mu^\eps \leq \tau_1^\eps \},
\end{equation*}
that is, the defect is given by the stopping probability in one regeneration period.

Moment generating functions for first regeneration times are defined by
\begin{equation*}
 \phi^\eps(\rho) = \sum_{n=0}^\infty e^{\rho n} f^\eps(n), \ \rho \in \mathbb{R}.
\end{equation*}

We will assume that the distributions of first regeneration times satisfy the following conditions:
\begin{enumerate}
 \item[$\Ccontinuityreg$:]
 \begin{enumerate}
  \item[\textbf{(a)}] $f^\eps(n) \rightarrow f^\zero(n)$ as $\varepsilon \rightarrow 0$, for all $n = 0,1,\ldots,$ where the limiting distribution $f^\zero(n)$ is non-periodic and not concentrated at zero.
  \item[\textbf{(b)}] $f^\eps \rightarrow f^\zero \in [0,1)$ as $\varepsilon \rightarrow 0$.
 \end{enumerate}
\end{enumerate}
\begin{enumerate}
 \item[$\Cmgfreg$:] There exists $\delta > 0$ such that:
 \begin{enumerate}
  \item[\textbf{(a)}] $\limsup_{0 \leq \varepsilon \rightarrow 0} \phi^\eps(\delta) < \infty$.
  \item[\textbf{(b)}] $\phi^\zero(\delta) > 1$.
 \end{enumerate}
\end{enumerate}

The solution of the following characteristic equation plays a crucial role in what follows:
\begin{equation} \label{eq:chareqreg}
 \phi^\eps(\rho) = 1.
\end{equation}

Our first lemma gives some basic properties for the solution of Equation \eqref{eq:chareqreg}. The proof can be found in Silvestrov and Petersson (2013).

\begin{lemma} \label{lmm:chareqreg}
If conditions $\Ccontinuityreg$ and $\Cmgfreg$ hold, then there exists a unique non-negative solution $\rho^\eps$ of the characteristic equation \eqref{eq:chareqreg} for sufficiently small $\varepsilon$. Moreover, we have $\rho^\eps \rightarrow \rho^\zero < \delta$ as $\varepsilon \rightarrow 0$.
\end{lemma}

The root $\rho^\eps$ of the characteristic equation is only given as the solution of a non-linear equation. In order to give a more detailed description of the asymptotic behaviour of $\rho^\eps$ as $\varepsilon \rightarrow 0$ we can construct an asymptotic expansion. This requires some perturbation conditions on the following mixed power-exponential moment generating functions:
\begin{equation*}
 \phi^\eps(\rho,r) = \sum_{n=0}^\infty n^r e^{\rho n} f^\eps(n), \ \rho \in \mathbb{R}, \ r = 0,1,\ldots
\end{equation*}
Note that $\phi^\eps(\rho,0) = \phi^\eps(\rho)$.

It follows from condition $\Cmgfreg$ that there exist $\delta > 0$ and $\varepsilon_0 > 0$ such that $\phi^\eps(\delta) < \infty$ for all $\varepsilon \leq \varepsilon_0$. Using this, we get for all $\rho < \delta$, $r=0,1,\ldots,$ and $\varepsilon \leq \varepsilon_0$ that
\begin{equation*}
 \phi^\eps(\rho,r) \leq \left( \sup_{n \geq 0} n^r e^{-(\delta-\rho) n} \right) \phi^\eps(\delta) < \infty.
\end{equation*}

Let us now introduce our perturbation condition:
\begin{enumerate}
 \item[$\Cperturbationreg$:] $\phi^\eps(\rho^\zero,r) = \phi^\zero(\rho^\zero,r) + a_{1,r} \varepsilon + \cdots + a_{k-r,r} \varepsilon^{k-r} + o(\varepsilon^{k-r})$, for $r = 0,\ldots,k$, where $|a_{n,r}| < \infty$, $n = 1,\ldots,k-r$, $r = 0,\ldots,k$.
\end{enumerate}
For convenience we denote $a_{0,r} = \phi^\zero(\rho^\zero,r)$, for $r = 0,\ldots,k$.

In order to apply the theory of perturbed renewal equations, we also need the following condition:
\begin{enumerate}
 \item[$\Ctailprobreg$:] There exists $\gamma > 0$ such that
 \begin{equation*}
  \limsup_{0 \leq \varepsilon \rightarrow 0} \sum_{n=0}^\infty e^{(\rho^\zero + \gamma) n} q^\eps(n,\mathcal{X}) < \infty.
 \end{equation*}
\end{enumerate}

Furthermore, we define
\begin{equation*}
 \Gamma_0 = \{ A \in \Gamma : q^\eps(n,A) \rightarrow q^\zero(n,A) \ \text{as} \ \varepsilon \rightarrow 0, \ n=0,1,\ldots \}
\end{equation*}
and
\begin{equation*}
 \widetilde{\pi}^\zero(A) = \frac{\sum_{n=0}^\infty e^{\rho^\zero n} q^\zero(n,A)}{\sum_{n=0}^\infty n e^{\rho^\zero n} f^\zero(n)}.
\end{equation*}

Our first theorem shows how we can construct an asymptotic expansion for the root of the characteristic equation based on the coefficients given in condition $\Cperturbationreg$ and how this yields asymptotic exponential expansions for the probabilities $P^\eps(n,A)$, $A \in \Gamma_0$. This result is proved in Silvestrov, Petersson (2013) for a general renewal equation under slightly different conditions. In the following proof we show that the conditions in the present paper are sufficient in order to apply this result to prove Theorem \ref{thm:expexpreg}.

\begin{theorem} \label{thm:expexpreg}
Assume that conditions $\Ccontinuityreg$, $\Cmgfreg$, and $\Cperturbationreg$ hold.
\begin{enumerate}
\item[$\mathbf{(i)}$] Then, the root $\rho^\eps$ of the characteristic equation \eqref{eq:chareqreg} has the asymptotic expansion
\begin{equation*}
 \rho^\eps = \rho^\zero + c_1 \varepsilon + \cdots + c_k \varepsilon^k + o(\varepsilon^k),
\end{equation*}
where $c_1 = - a_{1,0} / a_{0,1}$ and for $n = 2,\ldots,k$,
\begin{equation*}
\begin{split}
 c_n = - \frac{1}{a_{0,1}} &\Bigg( a_{n,0} + \sum_{q=1}^{n-1} a_{n-q,1} c_q \\
 &+ \sum_{m=2}^n \sum_{q=m}^n a_{n-q,m} \cdot \sum_{n_1,\ldots,n_{q-1} \in D_{m,q}} \prod_{p=1}^{q-1} \frac{c_p^{n_p}}{n_p!} \Bigg),
\end{split}
\end{equation*}
with $D_{m,q}$ being the set of all non-negative integer solutions to the system
\begin{equation*}
 n_1 + \cdots + n_{q-1} = m, \quad n_1 + \cdots + (q-1) n_{q-1} = q.
\end{equation*}
\item[$\mathbf{(ii)}$] If, in addition, condition $\Ctailprobreg$ holds, then for any non-negative integer valued function $n^\eps \rightarrow \infty$ as $\varepsilon \rightarrow 0$ in such a way that $\varepsilon^r n^\eps \rightarrow \lambda_r \in [0,\infty)$ for some $1 \leq r \leq k$, we have
\begin{equation*}
 \frac{P^\eps(n^\eps,A)}{\exp( -(\rho^\zero + c_1 \varepsilon + \cdots + c_{r-1} \varepsilon^{r-1}) n^\eps )} \rightarrow \frac{\widetilde{\pi}^\zero(A)}{e^{\lambda_r c_r}} \ \text{as} \ \varepsilon \rightarrow 0, \ A \in \Gamma_0.
\end{equation*}
\end{enumerate}
\end{theorem}

\begin{proof}
It follows directly from a result given in Silvestrov and Petersson (2013) that part $\mathbf{(i)}$ holds. Furthermore, it also follows from this result that part $\mathbf{(ii)}$ holds for any $A \in \Gamma$ satisfying the following statements:
\begin{itemize}
\item[$\boldsymbol{(\alpha)}$] $\limsup_{0 \leq \varepsilon \rightarrow 0} |q^\eps(n,A)| < \infty$, for all $n=0,1,\ldots$
\item[$\boldsymbol{(\beta)}$] $\sum_{n=0}^\infty e^{\rho^\eps n} q^\eps(n,A) \rightarrow \sum_{n=0}^\infty e^{\rho^\zero n} q^\zero(n,A)$, as $\varepsilon \rightarrow 0$.
\item[$\boldsymbol{(\gamma)}$] $\limsup_{0 \leq \varepsilon \rightarrow 0} \sum_{n=0}^\infty e^{(\rho^\zero + \gamma)n} |q^\eps(n,A)| < \infty$, for some $\gamma > 0$.
\end{itemize}

Since we always have $0 \leq q^\eps(n,A) \leq 1$, it follows that statement $\boldsymbol{(\alpha)}$ holds for any $A \in \Gamma$. Also $\boldsymbol{(\gamma)}$ holds for any $A \in \Gamma$. This follows from condition $\Ctailprobreg$ since $0 \leq q^\eps(n,A) \leq q^\eps(n,\mathcal{X})$.

Let us finally show that $\boldsymbol{(\beta)}$ holds for any $A \in \Gamma_0$.

It follows from Lemma \ref{lmm:chareqreg} that for every $\beta > 0$, we have $\rho^\eps \leq \rho^\zero + \beta$ for sufficiently small $\varepsilon$. Let us choose $\beta$ such that $0 < \beta < \gamma$, where $\gamma$ is the value from condition $\Ctailprobreg$. Then,
\begin{equation} \label{eq:taillimit}
\begin{split}
 &\lim_{N \rightarrow \infty} \limsup_{0 \leq \varepsilon \rightarrow 0} \sum_{n=N+1}^\infty e^{\rho^\eps n} q^\eps(n,A) \\
 &\quad \quad \leq \lim_{N \rightarrow \infty} \limsup_{0 \leq \varepsilon \rightarrow 0} \sum_{n=N+1}^\infty e^{(\rho^\zero + \beta) n} q^\eps(n,\mathcal{X}) \\
 &\quad \quad \leq \lim_{N \rightarrow \infty} e^{-(\gamma - \beta)(N+1)} \left( \limsup_{0 \leq \varepsilon \rightarrow 0} \sum_{n=0}^\infty e^{(\rho^\zero + \gamma) n} q^\eps(n,\mathcal{X}) \right) = 0.
\end{split}
\end{equation}

It now follows from \eqref{eq:taillimit}, Lemma \ref{lmm:chareqreg}, and the definition of $\Gamma_0$ that for any $A \in \Gamma_0$,
\begin{equation*}
 \lim_{\varepsilon \rightarrow 0} \sum_{n=0}^\infty e^{\rho^\eps n} q^\eps(n,A) = \lim_{N \rightarrow \infty} \lim_{\varepsilon \rightarrow 0} \sum_{n=0}^N e^{\rho^\eps n} q^\eps(n,A) = \sum_{n=0}^\infty e^{\rho^\zero n} q^\zero(n,A).
\end{equation*}
\end{proof}

\section{Perturbed Semi-Markov Processes} \label{sec:semimarkov}

In this section we define perturbed discrete time semi-Markov processes.

For every $\varepsilon \geq 0$, let $(\eta_n^\eps,\kappa_n^\eps), \ n=0,1,\ldots,$ be a discrete time Markov chain on the state space $(X,\mathbb{N})$, where $X = \{ 0,1,\ldots,N \}$ and $\mathbb{N} = \{ 1,2,\ldots \}$. We assume that the Markov chain is homogeneous in time and that the transition probabilities do not depend on the current value of the second component. Thus, the process $(\eta_n^\eps,\kappa_n^\eps)$ is characterized by an initial distribution $p_i^\eps = \mathsf{P} \{ \eta_0^\eps = i \}$, $i \in X$, and transition probabilities
\begin{equation*}
 Q_{i j}^\eps(k) = \mathsf{P} \{ \eta_{n+1}^\eps = j, \ \kappa_{n+1}^\eps = k \cond \eta_n^\eps = i \}, \ i,j \in X, \ k \in \mathbb{N}.
\end{equation*}

Let $\tau^\eps(0) = 0$ and $\tau^\eps(n) = \kappa_1^\eps + \cdots + \kappa_n^\eps$ for $n \geq 1$. Furthermore, let $\nu^\eps(n) = \max \{ k \geq 0 : \tau^\eps(k) \leq n \}$ for $n \geq 0$. The semi-Markov process associated with the Markov chain $(\eta_n^\eps,\kappa_n^\eps)$ is defined by
\begin{equation*}
 \xi^\eps(n) = \eta_{\nu^\eps(n)}^\eps, \ n=0,1,\ldots
\end{equation*}

For the semi-Markov process $\xi^\eps(n)$, we have that $\kappa_n^\eps$ are the times between successive moments of jumps, $\tau^\eps(n)$ are the moments of the jumps, and $\nu^\eps(n)$ are the number of jumps in the interval $[0,n]$.

Since the transition probabilities of the Markov chain $(\eta_n^\eps,\kappa_n^\eps)$ do not depend on the current value of the second component, it follows that $\eta_n^\eps$ is itself a (homogeneous) Markov chain. Its transition probabilities are given by
\begin{equation*}
 p_{i j}^\eps = \sum_{k=1}^\infty Q_{i j}^\eps(k) = \mathsf{P} \{ \eta_{n+1}^\eps = j \cond \eta_n^\eps = i \}, \ i,j \in X,
\end{equation*}
and it is called an embedded Markov chain for the corresponding semi-Markov process.

It is sometimes convenient to write the transition probabilities of the Markov chain $(\eta_n^\eps,\kappa_n^\eps)$ as
\begin{equation*}
 Q_{i j}^\eps(k) = p_{i j}^\eps f_{i j}^\eps(k), \ i,j \in X, \ k \in \mathbb{N},
\end{equation*}
where
\begin{equation*}
 f_{i j}^\eps(k) = \mathsf{P} \{ \kappa_{n+1}^\eps = k \cond \eta_n^\eps = i, \ \eta_{n+1}^\eps = j \}
\end{equation*}
are the distributions of transition times.

Let us also define random variables for first hitting times. Let $\nu_j^\eps = \min \{ n \geq 1 : \eta_n^\eps = j \}$ and let $\mu_j^\eps = \tau^\eps(\nu_j^\eps)$. Then, $\nu_j^\eps$ is the first hitting time of the embedded Markov chain into state $j$ and $\mu_j^\eps$ is the first hitting time of the semi-Markov process into state $j$. Note that $\nu_j^\eps$ and $\mu_j^\eps$ are both possibly improper random variables taking values in the set $\mathbb{N} \cup \{ \infty \}$. Throughout the paper, we will use the notation
\begin{equation*}
 g_{i j}^\eps(n) = \mathsf{P}_i \{ \mu_j^\eps = n, \ \nu_0^\eps > \nu_j^\eps \}, \ i,j \in X, \ n=0,1,\ldots,
\end{equation*}
and
\begin{equation*}
 g_{i j}^\eps = \mathsf{P}_i \{ \nu_0^\eps > \nu_j^\eps \}, \ i,j \in X.
\end{equation*}
Here, and in what follows, we write $\mathsf{P}_i(A) = \mathsf{P}(A \cond \eta_0^\eps = i)$ for any event $A$. Corresponding notation for conditional expectations will also be used.

%We will study the model of perturbed semi-Markov processes with absorption. For the rest of the paper it will be assumed that state $0$ is an absorbing state for the semi-Markov process, that is, $p_{0 j}^\eps = 0$ for all $j \neq 0$ and $\varepsilon \geq 0$.

In order to consider the semi-Markov process $\xi^\eps(n)$, for $\varepsilon > 0$, as a perturbation of the semi-Markov process $\xi^\zero(n)$, the following continuity condition will be used:
\begin{enumerate}
 \item[$\Ccontinuity$:]
 \begin{enumerate}
  \item[\textbf{(a)}] $p_{i j}^\eps \rightarrow p_{i j}^\zero$ as $\varepsilon \rightarrow 0$, for all $i \neq 0$, $j \in X$.
  \item[\textbf{(b)}] $f_{i j}^\eps(n) \rightarrow f_{i j}^\zero(n)$ as $\varepsilon \rightarrow 0$, for all $i \neq 0$, $j \in X$, $n \in \mathbb{N}$.
 \end{enumerate}
\end{enumerate}

Furthermore, we will assume that $\{ 1,\ldots,N \}$ is a communicating class of states for sufficiently small $\varepsilon$. This is implied by condition $\Ccontinuity$ together with the following condition:
\begin{enumerate}
 \item[$\Cergodicity$:] $g_{i j}^\zero > 0$, for all $i,j \neq 0$.
\end{enumerate}
Transitions to state $0$ may, or may not be possible, both for the limiting process and the perturbed process.

\section{Moments of First Hitting Times} \label{sec:moments}

In this section we consider moment generating functions of first hitting times. First, a system of linear equations for these moment generating functions are derived and then, a necessary and sufficient condition for them to be finite is given.

Moment generating functions of first hitting times are defined by
\begin{equation*}
 \phi_{i j}^\eps(\rho) = \mathsf{E}_i e^{\rho \mu_j^\eps} \chi( \nu_0^\eps > \nu_j^\eps ), \ \rho \in \mathbb{R}, \ i,j \in X.
\end{equation*}
Alternatively, this can be written as
\begin{equation*}
 \phi_{i j}^\eps(\rho) = \sum_{n=0}^\infty e^{\rho n} g_{i j}^\eps(n), \ \rho \in \mathbb{R}, \ i,j \in X.
\end{equation*}

We also define moment generating functions for transition probabilities:
\begin{equation*}
 p_{i j}^\eps(\rho) = \sum_{n=0}^\infty e^{\rho n} Q_{i j}^\eps(n), \ \rho \in \mathbb{R}, \ i,j \in X.
\end{equation*}

By conditioning on $(\eta_1^\eps,\kappa_1^\eps)$ we get for any $i,j \neq 0$,
\begin{equation} \label{eq:phimom}
\begin{split}
 \phi_{i j}^\eps(\rho) &= \sum_{l = 0}^N \sum_{k=1}^\infty \mathsf{E}_i ( e^{\rho \mu_j^\eps} \chi( \nu_0^\eps > \nu_j^\eps ) \cond \eta_1^\eps = l, \kappa_1^\eps = k ) Q_{i l}^\eps(k) \\
 &= \sum_{k=1}^\infty e^{\rho k} Q_{i j}(k) + \sum_{l \neq 0,j} \sum_{k=1}^\infty \mathsf{E}_l e^{\rho (k + \mu_j^\eps)} \chi( \nu_0^\eps > \nu_j^\eps ) Q_{i l}^\eps(k) \\
 &= p_{i j}^\eps(\rho) + \sum_{l \neq 0,j} p_{i l}^\eps(\rho) \phi_{l j}^\eps(\rho).
\end{split}
\end{equation}

Throughout the paper we will use the convention $0 \cdot \infty = 0$. With this convention, relation \eqref{eq:phimom} holds for all $\rho \in \mathbb{R}$ and $i,j \neq 0$, even in the case where some of the moment generating functions involved take infinite values. In this case relation \eqref{eq:phimom} may take the form $\infty = \infty$.

In what follows, it will sometimes be more convenient to work with matrices. For each $j \neq 0$, we define column vectors
\begin{equation} \label{eq:vectorphimom}
 \mathsf{\Phi}_j^\eps(\rho) = \begin{bmatrix} \phi_{1 j}^\eps(\rho) & \phi_{2 j}^\eps(\rho) & \cdots & \phi_{N j}^\eps(\rho) \end{bmatrix}^T,
\end{equation}
\begin{equation} \label{eq:vectorpmom}
 \mathbf{p}_j^\eps(\rho) = \begin{bmatrix} p_{1 j}^\eps(\rho) & p_{2 j}^\eps(\rho) & \cdots & p_{N j}^\eps(\rho) \end{bmatrix}^T,
\end{equation}
and $N \times N$ matrices $\jP^\eps(\rho) = \| \jp_{i k}^\eps(\rho) \|$ where the elements are given by
\begin{equation} \label{eq:matrixjPmom}
 \jp_{i k}^\eps(\rho) = \left\{
 \begin{array}{l l}
  p_{i k}^\eps(\rho) & i=1,\ldots,N, \ k \neq j, \\
  0 & i=1,\ldots,N, \ k = j.
 \end{array}
 \right.
\end{equation}

Using \eqref{eq:vectorphimom}, \eqref{eq:vectorpmom}, and \eqref{eq:matrixjPmom}, we can write \eqref{eq:phimom} in matrix notation:
\begin{equation} \label{eq:systemphimom}
 \mathsf{\Phi}_j^\eps(\rho) = \mathbf{p}_j^\eps(\rho) + \jP^\eps(\rho) \mathsf{\Phi}_j^\eps(\rho), \ j \neq 0.
\end{equation}
The vectors and matrices above are allowed to have entries with the value $\infty$. By remarks given above, this means that relation \eqref{eq:systemphimom} holds for all $\rho \in \mathbb{R}$.

We will now derive an alternative representation for the vector $\mathsf{\Phi}_j^\eps(\rho)$ of moment generating functions.

Let us for each $j \neq 0$ define an $N \times N$ matrix valued function $\jA^\eps(\rho) = \| \ja_{i k}^\eps(\rho) \|$ by
\begin{equation} \label{eq:defjA}
 \jA^\eps(\rho) = \mathbf{I} + \jP^\eps(\rho) + ( \jP^\eps(\rho) )^2 + \cdots, \ \rho \in \mathbb{R}.
\end{equation}
Since all elements of the matrices on the right hand side are non-negative, it follows that $\jA^\eps(\rho)$ is well defined and has elements that take values in the set $[0,\infty]$. As will be shown next, the elements of $\jA^\eps(\rho)$ can be given a probabilistic interpretation.

Let $j \neq 0$ be fixed. We define random variables by
\begin{equation} \label{eq:defdeltajk}
 \delta_{j k}^\eps(\rho) = \sum_{n=0}^\infty e^{\rho \tau^\eps(n)} \chi( \nu_0^\eps \wedge \nu_j^\eps > n, \ \eta_n^\eps = k ), \ k \neq 0.
\end{equation}
Notice that $\delta_{j j}^\eps(\rho) = \chi( \eta_0^\eps = j )$.

For $n=1,2,\ldots,$ we have
\begin{equation} \label{eq:meanoneterm}
\begin{split}
 &\mathsf{E}_i e^{\rho \tau^\eps(n)} \chi( \nu_0^\eps \wedge \nu_j^\eps > n, \ \eta_n^\eps = k ) \\
 &\quad \quad = \sum_{\substack{i_0=i; \, i_n=k; \\ i_1,\ldots,i_{n-1} \neq 0,j}} \mathsf{E}_i ( e^{\rho \tau^\eps(n)} \cond \eta_1^\eps = i_1,\ldots, \eta_n^\eps = i_n ) \prod_{m=1}^n p_{i_{m-1} i_m}^\eps \\
 &\quad \quad = \sum_{\substack{i_0=i; \, i_n=k; \\ i_1,\ldots,i_{n-1} \neq 0,j}} \prod_{m=1}^n p_{i_{m-1} i_m}^\eps(\rho), \ i \neq 0, \ k \neq 0,j.
\end{split}
\end{equation}

From \eqref{eq:defjA}, \eqref{eq:defdeltajk}, and \eqref{eq:meanoneterm} it follows that
\begin{equation} \label{eq:probinterpret}
 \ja_{i k}^\eps(\rho) = \mathsf{E}_i \delta_{j k}^\eps(\rho), \ i,k \neq 0.
\end{equation}

Let us now derive an alternative formula for $\mathsf{\Phi}_j^\eps(\rho)$.

By definition we have
\begin{equation} \label{eq:defphi}
 \phi_{i j}^\eps(\rho) = \mathsf{E}_i e^{\rho \mu_j^\eps} \chi( \nu_0^\eps > \nu_j^\eps ), \ i,j \neq 0.
\end{equation}

The indicator function can be written as
\begin{equation} \label{eq:indphi}
\begin{split}
 &\chi( \nu_0^\eps > \nu_j^\eps ) = \chi(\eta_1^\eps = j) \\
 &\quad \quad + \sum_{n=1}^\infty \sum_{k \neq 0,j} \chi( \nu_0^\eps \wedge \nu_j^\eps > n, \ \eta_n^\eps = k, \ \eta_{n+1}^\eps = j ).
\end{split}
\end{equation}

Note that for all $i,j \neq 0$,
\begin{equation} \label{eq:altformulaphi1}
 \mathsf{E}_i e^{\rho \mu_j^\eps} \chi( \eta_1^\eps = j ) = \mathsf{E}_i( e^{\rho \kappa_1^\eps} \cond \eta_1^\eps = j ) p_{i j}^\eps = p_{i j}^\eps(\rho)
\end{equation}
and
\begin{equation} \label{eq:altformulaphi2}
\begin{split}
 &\sum_{n=1}^\infty \sum_{k \neq 0,j} \mathsf{E}_i e^{\rho \mu_j^\eps} \chi( \nu_0^\eps \wedge \nu_j^\eps > n, \ \eta_n^\eps = k, \ \eta_{n+1}^\eps = j ) \\
 &\quad \quad = \sum_{n=1}^\infty \sum_{k \neq 0,j} \mathsf{E}_i e^{\rho \tau^\eps(n)} \chi( \nu_0^\eps \wedge \nu_j^\eps > n, \ \eta_n^\eps = k ) p_{k j}^\eps(\rho).
\end{split}
\end{equation}

From \eqref{eq:defdeltajk} and \eqref{eq:defphi}--\eqref{eq:altformulaphi2} it follows that for all $i,j \neq 0$,
\begin{equation} \label{eq:altformulaphi}
\begin{split}
 \phi_{i j}^\eps(\rho) &= \sum_{n=0}^\infty \sum_{k \neq 0} \mathsf{E}_i e^{\rho \tau^\eps(n)} \chi( \nu_0^\eps \wedge \nu_j^\eps > n, \ \eta_n^\eps = k ) p_{k j}^\eps(\rho) \\
 &= \sum_{k \neq 0} p_{k j}^\eps(\rho) \mathsf{E}_i \delta_{j k}^\eps(\rho).
\end{split}
\end{equation}

Now using \eqref{eq:probinterpret} we can write \eqref{eq:altformulaphi} in matrix notation:
\begin{equation} \label{eq:phimatrixalt}
 \mathsf{\Phi}_j^\eps(\rho) = \jA^\eps(\rho) \mathbf{p}_j^\eps(\rho), \ \rho \in \mathbb{R}, \ j \neq 0.
\end{equation}
This representation will now be used to prove the following lemma which gives a necessary and sufficient condition for $\mathsf{\Phi}_j^\eps(\rho)$ to be finite.

\begin{lemma} \label{lmm:finitemgfs}
Assume that for some $\varepsilon \geq 0$ we have $g_{i j}^\eps > 0$, for all $i,j \neq 0$. Then $\mathsf{\Phi}_j^\eps(\rho) < \infty$ if and only if $\mathbf{p}_j^\eps(\rho) < \infty$, $\jP^\eps(\rho) < \infty$, and the inverse matrix $( \mathbf{I} - \jP^\eps(\rho) )^{-1}$ exists.
\end{lemma}

\begin{proof}
Let us first assume that $\mathsf{\Phi}_j^\eps(\rho) < \infty$.

Since $g_{i j}^\eps > 0$ for all $i,j \neq 0$, it follows from \eqref{eq:phimatrixalt} that $\jA^\eps(\rho)$ and $\mathbf{p}_j^\eps(\rho)$ are finite. Moreover, it follows from the definition of $\jA^\eps(\rho)$ that $\jP^\eps(\rho) < \infty$ if $\jA^\eps(\rho) < \infty$, so we have
\begin{equation} \label{eq:finitemgfs1}
 \mathbf{p}_j^\eps(\rho), \ \jP^\eps(\rho), \ \jA^\eps(\rho) < \infty.
\end{equation}
The definition of $\jA^\eps(\rho)$ also yields
\begin{equation} \label{eq:finitemgfs2}
\begin{split}
 \jA^\eps(\rho) &= \mathbf{I} + \jP^\eps(\rho) \left( \mathbf{I} + \jP^\eps(\rho) + (\jP^\eps(\rho))^2 + \cdots \right) \\
 &= \mathbf{I} + \jP^\eps(\rho) \jA^\eps(\rho).
\end{split}
\end{equation}

It follows from \eqref{eq:finitemgfs1} that we can rewrite \eqref{eq:finitemgfs2} as
\begin{equation*}
 \mathbf{I} = ( \mathbf{I} - \jP^\eps(\rho) ) \jA^\eps(\rho).
\end{equation*}
This means that $( \mathbf{I} - \jP^\eps(\rho) )$ has an inverse matrix given by $\jA^\eps(\rho)$.

Now assume that $\mathbf{p}_j^\eps(\rho) < \infty$, $\jP^\eps(\rho) < \infty$, and that the inverse matrix $( \mathbf{I} - \jP^\eps(\rho) )^{-1}$ exists.

First note that then the following relation holds:
\begin{equation} \label{eq:finitemgfs3}
 ( \mathbf{I} - \jP^\eps(\rho) )^{-1} = \mathbf{I} + \jP^\eps(\rho) ( \mathbf{I} - \jP^\eps(\rho) )^{-1}.
\end{equation}

Iterating Equation \eqref{eq:finitemgfs3} gives for $n=1,2,\ldots,$
\begin{equation} \label{eq:finitemgfs4}
\begin{split}
 &( \mathbf{I} - \jP^\eps(\rho) )^{-1} = \mathbf{I} + \jP^\eps(\rho) + \cdots + ( \jP^\eps(\rho) )^n \\
 &\quad \quad + ( \jP^\eps(\rho) )^{n+1} ( \mathbf{I} - \jP^\eps(\rho) )^{-1}.
\end{split}
\end{equation}

Since $( \mathbf{I} - \jP^\eps(\rho) )^{-1} < \infty$ it follows from \eqref{eq:finitemgfs4} that we necessarily have
\begin{equation} \label{eq:finitemgfs5}
 ( \jP^\eps(\rho) )^{n+1} ( \mathbf{I} - \jP^\eps(\rho) )^{-1} \rightarrow \mathbf{0}, \ \text{as} \ n \rightarrow \infty.
\end{equation}

Letting $n \rightarrow \infty$ in \eqref{eq:finitemgfs4} and using \eqref{eq:finitemgfs5}, it follows that 
\begin{equation} \label{eq:finitemgfs6}
 \jA^\eps(\rho) = ( \mathbf{I} - \jP^\eps(\rho) )^{-1} < \infty.
\end{equation}

From \eqref{eq:phimatrixalt} and \eqref{eq:finitemgfs6} we conclude that $\mathsf{\Phi}_j^\eps(\rho) < \infty$.
\end{proof}

We supplement Lemma \ref{lmm:finitemgfs} with a corresponding result for the moment generating functions
\begin{equation*}
 \widetilde{\phi}_{i j}^\eps(\rho) = \mathsf{E}_i e^{\rho \mu_0^\eps} \chi( \nu_0^\eps < \nu_j^\eps ), \ \rho \in \mathbb{R}, \ i,j \neq 0.
\end{equation*}

Similar calculations as above show that we have the representation
\begin{equation*}
 \widetilde{\mathsf{\Phi}}_j^\eps(\rho) = \jA^\eps(\rho) \mathbf{p}_0^\eps(\rho), \ \rho \in \mathbb{R}, \ j \neq 0,
\end{equation*}
where
\begin{equation*}
 \widetilde{\mathsf{\Phi}}_j^\eps(\rho) = \begin{bmatrix} \widetilde{\phi}_{1 j}^\eps(\rho) & \widetilde{\phi}_{2 j}^\eps(\rho) & \ldots & \widetilde{\phi}_{N j}^\eps(\rho) \end{bmatrix}^T
\end{equation*}
and
\begin{equation*}
 \mathbf{p}_0^\eps(\rho) = \begin{bmatrix} p_{1 0}^\eps(\rho) & p_{2 0}^\eps(\rho) & \ldots & p_{N 0}^\eps(\rho) \end{bmatrix}^T.
\end{equation*}

The following lemma gives a necessary and sufficient condition for $\widetilde{\mathsf{\Phi}}_j^\eps(\rho)$ to be finite. The proof is analogous to the proof of Lemma \ref{lmm:finitemgfs} and is therefore omitted.

\begin{lemma} \label{lmm:finitemgfs2}
Assume that for some $\varepsilon \geq 0$ we have $g_{i j}^\eps > 0$, for all $i,j \neq 0$. Then $\widetilde{\mathsf{\Phi}}_j^\eps(\rho) < \infty$ if and only if $\mathbf{p}_0^\eps(\rho) < \infty$, $\jP^\eps(\rho) < \infty$, and the inverse matrix $( \mathbf{I} - \jP^\eps(\rho) )^{-1}$ exists.
\end{lemma}

\section{Solidarity Properties for Moments of First Hitting Times} \label{sec:solidarity}

In this section we first present a condition of Cram{\'e}r type for the distributions of transition times. Then, a solidarity lemma for moment generating functions of first hitting times is proved which motivates the specific form of this condition.

We define moment generating functions for transition times by
\begin{equation*}
 \psi_{i j}^\eps(\rho) = \sum_{n=0}^\infty e^{\rho n} f_{i j}^\eps(n), \ \rho \in \mathbb{R}, \ i,j \in X.
\end{equation*}
Notice that $p_{i j}^\eps(\rho) = p_{i j}^\eps \psi_{i j}^\eps(\rho)$.

For semi-Markov processes it is natural to formulate the Cram{\'e}r type condition corresponding to $\Cmgfreg$ in terms of moments of transition times:
\begin{equation} \label{eq:cramercond}
 \limsup_{0 \leq \varepsilon \rightarrow 0} \psi_{i j}^\eps(\beta) < \infty, \ i \neq 0, \ j \in X, \ \text{for some} \ \beta > 0.
\end{equation}

It can be shown that relation \eqref{eq:cramercond} together with conditions $\Ccontinuity$ and $\Cergodicity$ imply that part $\mathbf{(a)}$ of condition $\Cmgfreg$ holds for the moment generating functions $\phi_{i i}^\eps(\rho)$. However, it need not be that part $\mathbf{(b)}$ of condition $\Cmgfreg$ holds. In order to guarantee this, we will use the following condition:
\begin{enumerate}
 \item[$\Cmgf$:] There exists $\beta > 0$ such that:
 \begin{enumerate}
  \item[\textbf{(a)}] $\limsup_{0 \leq \varepsilon \rightarrow 0} \psi_{i j}^\eps(\beta) < \infty$, for all $i \neq 0$, $j \in X$.
  \item[\textbf{(b)}] $\phi_{i i}^\zero(\beta_i) > 1$, for some $i \neq 0$ and $\beta_i \leq \beta$.
 \end{enumerate}
\end{enumerate}

Let us introduce the following moment generating functions:
\begin{equation*}
 \kphi_{i j}^\eps(\rho) = \mathsf{E}_i e^{\rho \mu_j^\eps} \chi( \nu_0^\eps \wedge \nu_k^\eps > \nu_j^\eps ), \ \rho \in \mathbb{R}, \ i,j,k \in X.
\end{equation*}

Before giving the solidarity lemma, we first prove an auxiliary lemma which gives a connection between $\phi_{i i}^\eps(\rho)$ and $\phi_{j j}^\eps(\rho)$.

\begin{lemma} \label{lmm:solidarity}
Let $i \neq 0$ be fixed. Assume that we for some $\varepsilon \geq 0$ and $\rho \in \mathbb{R}$ have:
\begin{itemize}
 \item[$\boldsymbol{(\alpha)}$] $g_{k j}^\eps > 0$, for all $k,j \neq 0$.
 \item[$\boldsymbol{(\beta)}$] $\phi_{i i}^\eps(\rho) \leq 1$.
\end{itemize}
Then, the following relation holds for all $j \neq i$:
\begin{equation} \label{eq:solidarityid}
 ( 1 - \phi_{i i}^\eps(\rho) )( 1 - \iphi_{j j}^\eps(\rho) ) = ( 1 - \phi_{j j}^\eps(\rho) )( 1 - \jphi_{i i}^\eps(\rho) ).
\end{equation}
\end{lemma}

\begin{proof}
By using the regenerative property of the semi-Markov process we can for any $j \neq 0,i$ write the following relations for moment generating functions:
\begin{equation} \label{eq:solidaritylmmeq1}
 \phi_{i i}^\eps(\rho) = \jphi_{i i}^\eps(\rho) + \iphi_{i j}^\eps(\rho) \phi_{j i}^\eps(\rho),
\end{equation}
\begin{equation} \label{eq:solidaritylmmeq2}
 \phi_{j i}^\eps(\rho) = \jphi_{j i}^\eps(\rho) + \iphi_{j j}^\eps(\rho) \phi_{j i}^\eps(\rho),
\end{equation}
\begin{equation} \label{eq:solidaritylmmeq3}
 \phi_{j j}^\eps(\rho) = \iphi_{j j}^\eps(\rho) + \jphi_{j i}^\eps(\rho) \phi_{i j}^\eps(\rho),
\end{equation}
\begin{equation} \label{eq:solidaritylmmeq4}
 \phi_{i j}^\eps(\rho) = \iphi_{i j}^\eps(\rho) + \jphi_{i i}^\eps(\rho) \phi_{i j}^\eps(\rho).
\end{equation}
Recall the we use the convention $0 \cdot \infty = 0$, so relations \eqref{eq:solidaritylmmeq1}--\eqref{eq:solidaritylmmeq4} hold for all $\rho \in \mathbb{R}$.

It follows from $\boldsymbol{(\alpha)}$ that
\begin{equation} \label{eq:solidaritylmmeq5a}
 \phi_{i i}^\eps(\rho), \ \phi_{i j}^\eps(\rho), \ \phi_{j i}^\eps(\rho), \ \phi_{j j}^\eps(\rho), \ \iphi_{i j}^\eps(\rho), \ \jphi_{j i}^\eps(\rho) \in (0,\infty].
\end{equation}

From $\boldsymbol{(\beta)}$, \eqref{eq:solidaritylmmeq1}, and \eqref{eq:solidaritylmmeq5a} we can conclude that
\begin{equation} \label{eq:solidaritylmmeq5}
 \phi_{i i}^\eps(\rho), \ \phi_{j i}^\eps(\rho), \ \iphi_{i j}^\eps(\rho), \ \jphi_{i i}^\eps(\rho) < \infty.
\end{equation}
Furthermore, it follows from, \eqref{eq:solidaritylmmeq2}, \eqref{eq:solidaritylmmeq5a}, and \eqref{eq:solidaritylmmeq5} that
\begin{equation} \label{eq:solidaritylmmeq6}
 \iphi_{j j}^\eps(\rho), \ \jphi_{j i}^\eps(\rho) < \infty.
\end{equation}

Thus, all generating functions in Equations \eqref{eq:solidaritylmmeq1} and \eqref{eq:solidaritylmmeq2} are finite under conditions $\boldsymbol{(\alpha)}$ and $\boldsymbol{(\beta)}$. However, it is not immediate that also $\phi_{i j}^\eps(\rho)$ and $\phi_{j j}^\eps(\rho)$ are finite. In order to prove this, let us consider random variables for successive return times. We define the the $n$-th return to a state $j$ for the embedded Markov chain by $\nu_j^\eps(0) = 0$ and
\begin{equation*}
 \nu_j^\eps(n) = \min \{ k > \nu_j^\eps(n-1) : \eta_k^\eps = j \}, \ n=1,2,\ldots
\end{equation*} 
Corresponding return times for the semi-Markov process are defined by
\begin{equation*}
 \mu_j^\eps(n) = \tau^\eps( \nu_j^\eps(n) ), \ n=0,1,\ldots
\end{equation*}

Using the variables for return times, we can write
\begin{equation} \label{eq:solidaritylmmeq7}
\begin{split}
 &\chi( \nu_0^\eps > \nu_j^\eps ) = \chi( \nu_0^\eps \wedge \nu_i^\eps > \nu_j^\eps ) \\
 &\quad \quad + \sum_{n=1}^\infty \chi( \nu_0^\eps \wedge \nu_j^\eps > \nu_i^\eps(n), \ \nu_0^\eps \wedge \nu_i^\eps(n+1) > \nu_j^\eps ).
\end{split}
\end{equation}

For $n=1,2,\ldots,$ it follows from the regenerative property of the semi-Markov process that
\begin{equation} \label{eq:solidaritylmmeq8}
\begin{split}
 &\mathsf{E}_i e^{\rho \mu_j^\eps} \chi( \nu_0^\eps \wedge \nu_j^\eps > \nu_i^\eps(n), \ \nu_0^\eps \wedge \nu_i^\eps(n+1) > \nu_j^\eps ) \\
 &\quad \quad = \mathsf{E}_i e^{\rho \mu_i^\eps(n)} \chi( \nu_0^\eps \wedge \nu_j^\eps > \nu_i^\eps(n) ) \mathsf{E}_i e^{\rho \mu_j^\eps} \chi( \nu_0^\eps \wedge \nu_i^\eps > \nu_j^\eps ).
\end{split}
\end{equation}

Using \eqref{eq:solidaritylmmeq7} and \eqref{eq:solidaritylmmeq8} we obtain
\begin{equation} \label{eq:solidaritylmmeq9}
 \phi_{i j}^\eps(\rho) = \iphi_{i j}^\eps(\rho) + \sum_{n=1}^\infty ( \jphi_{i i}^\eps(\rho) )^n \iphi_{i j}^\eps(\rho).
\end{equation}

It follows from \eqref{eq:solidaritylmmeq1}, \eqref{eq:solidaritylmmeq5a}, \eqref{eq:solidaritylmmeq5}, and $\boldsymbol{(\beta)}$ that $\jphi_{i i}^\eps(\rho) < 1$. Using \eqref{eq:solidaritylmmeq5}, \eqref{eq:solidaritylmmeq9}, and $\jphi_{i i}^\eps(\rho) < 1$ it follows that $\phi_{i j}^\eps(\rho) < \infty$. Then, we can use \eqref{eq:solidaritylmmeq3}, \eqref{eq:solidaritylmmeq6}, and $\phi_{i j}^\eps(\rho) < \infty$ to conclude that $\phi_{j j}^\eps(\rho) < \infty$.

It has now been shown that all generating functions in \eqref{eq:solidaritylmmeq1}--\eqref{eq:solidaritylmmeq4} are finite and these relations can now be used to prove that \eqref{eq:solidarityid} holds.

We can rewrite \eqref{eq:solidaritylmmeq2} as
\begin{equation} \label{eq:solidaritylmmeq10}
 \phi_{j i}^\eps(\rho) ( 1 - \iphi_{j j}^\eps(\rho) ) = \jphi_{j i}^\eps(\rho),
\end{equation}
Multiplying \eqref{eq:solidaritylmmeq1} by $( 1 - \iphi_{j j}^\eps(\rho) )$ and using \eqref{eq:solidaritylmmeq10} we get
\begin{equation} \label{eq:solidaritylmmeq11}
 \phi_{i i}^\eps(\rho) ( 1 - \iphi_{j j}^\eps(\rho) ) = \jphi_{i i}^\eps(\rho) ( 1 - \iphi_{j j}^\eps(\rho) ) + \iphi_{i j}^\eps(\rho) \jphi_{j i}^\eps(\rho).
\end{equation}
Subtracting $( 1 - \iphi_{j j}^\eps(\rho) )$ from both sides in \eqref{eq:solidaritylmmeq11} and then changing signs yield
\begin{equation} \label{eq:solidaritylmmeq12}
\begin{split}
 &( 1 - \phi_{i i}^\eps(\rho) ) ( 1 - \iphi_{j j}^\eps(\rho) ) \\
 &\quad \quad = ( 1 - \jphi_{i i}^\eps(\rho) ) ( 1 - \iphi_{j j}^\eps(\rho) ) - \iphi_{i j}^\eps(\rho) \jphi_{j i}^\eps(\rho).
\end{split}
\end{equation}
Similarly, using \eqref{eq:solidaritylmmeq3} and \eqref{eq:solidaritylmmeq4} we obtain
\begin{equation} \label{eq:solidaritylmmeq13}
\begin{split}
 &( 1 - \phi_{j j}^\eps(\rho) ) ( 1 - \jphi_{i i}^\eps(\rho) ) \\
 &\quad \quad = ( 1 - \iphi_{j j}^\eps(\rho) ) ( 1 - \jphi_{i i}^\eps(\rho) ) - \jphi_{j i}^\eps(\rho) \iphi_{i j}^\eps(\rho).
\end{split}
\end{equation}
Relation \eqref{eq:solidarityid} now follows from \eqref{eq:solidaritylmmeq12} and \eqref{eq:solidaritylmmeq13}.
\end{proof}

The next lemma is essential for the proof of our main result. The form of part $\mathbf{(b)}$ of condition $\Cmgf$ implies that the results of this lemma can be considered as solidarity properties for moments of first hitting times.

\begin{lemma} \label{lmm:mgfproperties}
Assume that conditions $\Ccontinuity$, $\Cergodicity$, and $\Cmgf$ hold. Let $i \neq 0$ be the state and $0 < \beta_i \leq \beta$ the number in condition $\Cmgf$ for which we have $\phi_{i i}^\zero(\beta_i) > 1$. Then:
\begin{enumerate}
 \item[$\mathbf{(i)}$] There exists $\rho' \in [0,\beta_i)$ such that $\phi_{j j}^\zero(\rho') = 1$ for any $j \neq 0$.
 \item[$\mathbf{(ii)}$] For any $j \neq 0$, there exists $\beta_j \in (\rho',\beta_i]$ such that $\phi_{j j}^\zero(\beta_j) > 1$ and $\phi_{k j}^\zero(\beta_j) < \infty$ for all $k \neq 0$.
 \item[$\mathbf{(iii)}$] There exists $\delta \in (0,\beta]$ such that $\phi_{j j}^\zero(\delta) > 1$, $j \neq 0$ and $\phi_{k j}^\zero(\delta) < \infty$, $k,j \neq 0$.
 \item[$\mathbf{(iv)}$] There exists $\varepsilon_0 > 0$ such that for all $\varepsilon \leq \varepsilon_0$ we have $\phi_{j j}^\eps(\delta) > 1$, $j \neq 0$ and $\phi_{k j}^\eps(\delta) < \infty$, $k,j \neq 0$.
\end{enumerate}
\end{lemma}

\begin{proof}
It follows from conditions $\Cergodicity$ and $\Cmgf$ that $\phi_{i i}^\zero(\rho)$ is continuous and strictly increasing for $\rho \in [0,\beta_i]$. Moreover, $\phi_{i i}^\zero(0) = \mathsf{P}_i \{ \nu_0^\zero > \nu_i^\zero \} \leq 1$ and $\phi_{i i}^\zero(\beta_i) > 1$. From this it follows that there exists (a unique) $\rho' \in [0,\beta_i)$ such that
\begin{equation} \label{eq:mgflmmeq1}
 \phi_{i i}^\zero(\rho') = 1.
\end{equation}

Now, for any $j \neq 0,i$ we can write
\begin{equation} \label{eq:mgflmmeq2}
 \phi_{i i}^\zero(\rho') = \jphi_{i i}^\zero(\rho') + \iphi_{i j}^\zero(\rho') \phi_{j i}^\zero(\rho').
\end{equation}

We also notice that under condition $\Cergodicity$,
\begin{equation} \label{eq:mgflmmeq3}
 \iphi_{i j}^\zero(\rho'), \ \phi_{j i}^\zero(\rho') > 0.
\end{equation}

It follows from \eqref{eq:mgflmmeq1}, \eqref{eq:mgflmmeq2}, and \eqref{eq:mgflmmeq3} that
\begin{equation} \label{eq:mgflmmeq4}
 \jphi_{i i}^\zero(\rho') < 1.
\end{equation}

Applying Lemma \ref{lmm:solidarity} with $\varepsilon = 0$ and $\rho = \rho'$, and using \eqref{eq:mgflmmeq1} we get
\begin{equation} \label{eq:mgflmmeq5}
 ( 1 - \phi_{j j}^\zero(\rho') )( 1 - \jphi_{i i}^\zero(\rho') ) = 0.
\end{equation}

From \eqref{eq:mgflmmeq4} and \eqref{eq:mgflmmeq5} we conclude that $\phi_{j j}^\zero(\rho') = 1$ and this proves part $\mathbf{(i)}$ of the lemma.

We now prove part $\mathbf{(ii)}$.

Let $j \neq 0$ be arbitrary. It follows from part $\mathbf{(i)}$ that there exists $\rho' \in [0,\beta_i)$ such that $\phi_{j j}^\zero(\rho') = 1$. For any $k \neq 0,j$ we have
\begin{equation} \label{eq:mgflmmeq6}
 \phi_{j j}^\zero(\rho') = \kphi_{j j}^\zero(\rho') + \jphi_{j k}^\zero(\rho') \phi_{k j}^\zero(\rho').
\end{equation}
It follows from \eqref{eq:mgflmmeq6}, condition $\Cergodicity$, and $\phi_{j j}^\zero(\rho') = 1$ that
\begin{equation} \label{eq:mgflmmeq7}
 \phi_{k j}^\zero(\rho') < \infty, \ k \neq 0.
\end{equation}

Using \eqref{eq:mgflmmeq7} we can apply Lemma \ref{lmm:finitemgfs} to conclude that $\det( \mathbf{I} - \jP^\zero(\rho') ) \neq 0$. Under condition $\Cmgf$, the elements of the matrix $\jP^\zero(\rho)$ are continuous functions of $\rho \in [0,\beta]$. Since $\rho' < \beta_i \leq \beta$, we can find $\beta_j \in (\rho',\beta_i]$ such that $\det( \mathbf{I} - \jP^\zero(\beta_j) ) \neq 0$. Furthermore, it follows from condition $\Cmgf$ that $p_{k j}^\zero(\beta_j) < \infty$ for all $k,j \neq 0$, so by Lemma \ref{lmm:finitemgfs} we get
\begin{equation*}
 \phi_{k j}^\zero(\beta_j) < \infty, \ k \neq 0.
\end{equation*}
Also, since $\rho' < \beta_j$ and $\phi_{j j}^\zero(\rho') = 1$, we have $\phi_{j j}^\zero(\beta_j) > 1$ and this completes the proof of part $\mathbf{(ii)}$.

If we define $\delta = \min \{ \beta_j : j \neq 0 \}$, part $\mathbf{(iii)}$ follows from parts $\mathbf{(i)}$ and $\mathbf{(ii)}$.

Finally, let us prove part $\mathbf{(iv)}$.

By Equation \eqref{eq:systemphimom} we have that the vector $\mathsf{\Phi}_j^\eps(\delta)$ satisfies the following system of linear equations:
\begin{equation} \label{eq:mgflmmeq8}
 \mathsf{\Phi}_j^\eps(\delta) = \mathbf{p}_j^\eps(\delta) + \jP^\eps(\delta) \mathsf{\Phi}_j^\eps(\delta).
\end{equation}

From part $\mathbf{(iii)}$ and Lemma \ref{lmm:finitemgfs} it follows that
\begin{equation} \label{eq:mgflmmeq9}
 \det( \mathbf{I} - \jP^\zero(\delta) ) \neq 0, \ j \neq 0.
\end{equation}

From conditions $\Ccontinuity$ and $\Cmgf$ we get
\begin{equation} \label{eq:mgflmmeq10}
 p_{k j}^\eps(\delta) \rightarrow p_{k j}^\zero(\delta) < \infty, \ \text{as} \ \varepsilon \rightarrow 0, \ k,j \neq 0.
\end{equation}

It follows from \eqref{eq:mgflmmeq9} and \eqref{eq:mgflmmeq10} that we can find $\varepsilon_1 > 0$ such that for all $j \neq 0$ and $\varepsilon \leq \varepsilon_1$,
\begin{equation} \label{eq:mgflmmeq11}
 \det( \mathbf{I} - \jP^\eps(\delta) ) \neq 0, \quad \mathbf{p}_j^\eps(\delta) < \infty, \quad \jP^\eps(\delta) < \infty.
\end{equation}

From \eqref{eq:mgflmmeq11} and Lemma \ref{lmm:finitemgfs} we conclude that for any $j \neq 0$ and $\varepsilon \leq \varepsilon_1$ it holds that $\mathsf{\Phi}_j^\eps(\delta) < \infty$ and, moreover, $\mathsf{\Phi}_j^\eps(\delta)$ is the unique solution to the system of linear equations \eqref{eq:mgflmmeq8}, so we can write
\begin{equation} \label{eq:mgflmmeq12}
 \mathsf{\Phi}_j^\eps(\delta) = ( \mathbf{I} - \jP^\eps(\delta) )^{-1} \mathbf{p}_j^\eps(\delta), \ j \neq 0.
\end{equation}
Furthermore, it follows from \eqref{eq:mgflmmeq10} and \eqref{eq:mgflmmeq12} that for any $j \neq 0$, we have $\mathsf{\Phi}_j^\eps(\delta) \rightarrow \mathsf{\Phi}_j^\zero(\delta)$ as $\varepsilon \rightarrow 0$. In particular, for any $j \neq 0$ we have $\phi_{j j}^\eps(\delta) \rightarrow \phi_{j j}^\zero(\delta)$ as $\varepsilon \rightarrow 0$ and since $\phi_{j j}^\zero(\delta) > 1$, $j \neq 0$, this means that we can find $\varepsilon_2 > 0$ such that $\phi_{j j}^\eps(\delta) > 1$ for all $j \neq 0$ and $\varepsilon \leq \varepsilon_2$. It follows that with $\varepsilon_0 = \min \{ \varepsilon_1, \varepsilon_2 \}$, the claims of part $\mathbf{(iv)}$ hold and this concludes the proof of Lemma \ref{lmm:mgfproperties}.
\end{proof}

\section{Power Series Expansions for Moments of First Hitting Times} \label{sec:powerseries}

In this section it is shown how mixed power-exponential moments for first hitting times can be expanded in power series with respect to the perturbation parameter. We first derive recursive systems of linear equations for these moments. Then, some properties of asymptotic matrix expansions are presented. Finally, we construct the desired asymptotic expansions.

Mixed power-exponential moment generating functions of first hitting times are defined by
\begin{equation*}
 \phi_{i j}^\eps(\rho,r) = \mathsf{E}_i (\mu_j^\eps)^r e^{\rho \mu_j^\eps} \chi( \nu_0^\eps > \nu_j^\eps ), \ \rho \in \mathbb{R}, \ r=0,1,\ldots, \ i,j \in X.
\end{equation*}
Alternatively, this can be written as
\begin{equation*}
 \phi_{i j}^\eps(\rho,r) = \sum_{n=0}^\infty n^r e^{\rho n} g_{i j}^\eps(n), \ \rho \in \mathbb{R}, \ r=0,1,\ldots, \ i,j \in X.
\end{equation*}
Notice that $\phi_{i j}^\eps(\rho,0) = \phi_{i j}^\eps(\rho)$.

We also define mixed power-exponential moment generating functions for transition probabilities:
\begin{equation*}
 p_{i j}^\eps(\rho,r) = \sum_{n=0}^\infty n^r e^{\rho n} Q_{i j}^\eps(n), \ \rho \in \mathbb{R}, \ r=0,1,\ldots, \ i,j \in X.
\end{equation*}
Note that $p_{i j}^\eps(\rho,0) = p_{i j}^\eps(\rho)$. Also note that we can write $p_{i j}^\eps(\rho,r) = p_{i j}^\eps \psi_{i j}^\eps(\rho,r)$ where $p_{i j}^\eps$ are the transition probabilities for the embedded Markov chain and
\begin{equation*}
 \psi_{i j}^\eps(\rho,r) = \sum_{n=0}^\infty n^r e^{\rho n} f_{i j}^\eps(n), \ \rho \in \mathbb{R}, \ r=0,1,\ldots, \ i,j \in X.
\end{equation*}

It follows from condition $\Cmgf$ that there exist $\beta > 0$ and $\varepsilon_1 > 0$ such that
\begin{equation*}
 \sup_{\varepsilon \leq \varepsilon_1} \max_{\substack{i \neq 0 \\ j \in X}} \psi_{i j}^\eps(\beta) < \infty.
\end{equation*}
From this it follows that for all $i \neq 0$, $j \in X$, $\varepsilon \leq \varepsilon_1$, $\rho < \beta$, and $r = 0,1,\ldots,$ we have
\begin{equation*}
 p_{i j}^\eps(\rho,r) \leq \left( \sup_{n \geq 0} n^r e^{-(\beta - \rho)n} \right) p_{i j}^\eps \psi_{i j}^\eps(\beta) < \infty.
\end{equation*}

Under conditions $\Ccontinuity$, $\Cergodicity$, and $\Cmgf$, it is seen from Lemma \ref{lmm:mgfproperties} that there exist $\delta \in (0,\beta]$ and $\varepsilon_2 > 0$ such that
\begin{equation*}
 \sup_{\varepsilon \leq \varepsilon_2} \max_{i,j \neq 0} \phi_{i j}^\eps(\delta) < \infty.
\end{equation*}
Using this, we get for all $i,j \neq 0$, $\varepsilon \leq \varepsilon_2$, $\rho < \delta$, and $r = 0,1,\ldots,$
\begin{equation*}
 \phi_{i j}^\eps(\rho,r) \leq \left( \sup_{n \geq 0} n^r e^{-(\delta - \rho)n} \right) \phi_{i j}^\eps(\delta) < \infty.
\end{equation*}

Recall from Section \ref{sec:moments} that the moment generating functions of first hitting times satisfy the following relations:
\begin{equation} \label{eq:equationsphi0}
 \phi_{i j}^\eps(\rho) = p_{i j}^\eps(\rho) + \sum_{l \neq 0,j} p_{i l}^\eps(\rho) \phi_{l j}^\eps(\rho), \ i,j \neq 0.
\end{equation}

From the discussion above it follows that for any $i,j \neq 0$, $\varepsilon \leq \min \{ \varepsilon_1, \varepsilon_2 \}$, and $\rho < \delta$, the functions $p_{i j}^\eps(\rho)$ and $\phi_{i j}^\eps(\rho)$ are arbitrarily many times differentiable with respect to $\rho$. Moreover, the derivative of order $r$ for $p_{i j}^\eps(\rho)$ and $\phi_{i j}^\eps(\rho)$ are given by $p_{i j}^\eps(\rho,r)$ and $\phi_{i j}^\eps(\rho,r)$, respectively.

Differentiating both sides of relation \eqref{eq:equationsphi0} gives the following for all $\varepsilon \leq \min \{ \varepsilon_1, \varepsilon_2 \}$ and $\rho < \delta$:
\begin{equation} \label{eq:equationsphir1}
 \phi_{i j}^\eps(\rho,r) = \lambda_{i j}^\eps(\rho,r) + \sum_{l \neq 0,j} p_{i l}^\eps(\rho) \phi_{l j}^\eps(\rho,r), \ i,j \neq 0, \ r=1,2,\ldots,
\end{equation}
where
\begin{equation} \label{eq:equationsphir2}
 \lambda_{i j}^\eps(\rho,r) = p_{i j}^\eps(\rho,r) + \sum_{m=1}^r \binom{r}{m} \sum_{l \neq 0,j} p_{i l}^\eps(\rho,m) \phi_{l j}^\eps(\rho,r-m).
\end{equation}

Let us rewrite relations \eqref{eq:equationsphi0}, \eqref{eq:equationsphir1}, and \eqref{eq:equationsphir2} in matrix notation. For each $j \neq 0$, we define column vectors
\begin{equation} \label{eq:vectorphi}
 \mathsf{\Phi}_j^\eps(\rho,r) = \begin{bmatrix} \phi_{1 j}^\eps(\rho,r) & \phi_{2 j}^\eps(\rho,r) & \cdots & \phi_{N j}^\eps(\rho,r) \end{bmatrix}^T,
\end{equation}
\begin{equation} \label{eq:vectorlambda}
 \boldsymbol{\lambda}_j^\eps(\rho,r) = \begin{bmatrix} \lambda_{1 j}^\eps(\rho,r) & \lambda_{2 j}^\eps(\rho,r) & \cdots & \lambda_{N j}^\eps(\rho,r) \end{bmatrix}^T,
\end{equation}
\begin{equation} \label{eq:vectorp}
 \mathbf{p}_j^\eps(\rho,r) = \begin{bmatrix} p_{1 j}^\eps(\rho,r) & p_{2 j}^\eps(\rho,r) & \cdots & p_{N j}^\eps(\rho,r) \end{bmatrix}^T,
\end{equation}
and $N \times N$ matrices $\jP^\eps(\rho,r) = \| \jp_{i k}^\eps(\rho,r) \|$ where the elements are given by
\begin{equation} \label{eq:matrixjP}
 \jp_{i k}^\eps(\rho,r) = \left\{
 \begin{array}{l l}
  p_{i k}^\eps(\rho,r) & i=1,\ldots,N, \ k \neq j, \\
  0 & i=1,\ldots,N, \ k = j.
 \end{array}
 \right.
\end{equation}

With these definitions we have
\begin{equation} \label{eq:matricesmgf}
 \mathsf{\Phi}_j^\eps(\rho,0) = \mathsf{\Phi}_j^\eps(\rho), \ \mathbf{p}_j^\eps(\rho,0) = \mathbf{p}_j^\eps(\rho), \ \jP^\eps(\rho,0) = \jP^\eps(\rho).
\end{equation}

Using \eqref{eq:equationsphi0}--\eqref{eq:matricesmgf}, we get for $r = 0$,
\begin{equation} \label{eq:systemphi0}
 \mathsf{\Phi}_j^\eps(\rho) = \mathbf{p}_j^\eps(\rho) + \jP^\eps(\rho) \mathsf{\Phi}_j^\eps(\rho), \ j \neq 0,
\end{equation}
and for $r=1,2,\ldots,$
\begin{equation} \label{eq:systemphir1}
 \mathsf{\Phi}_j^\eps(\rho,r) = \boldsymbol{\lambda}_j^\eps(\rho,r) + \jP^\eps(\rho) \mathsf{\Phi}_j^\eps(\rho,r), \ j \neq 0,
\end{equation}
where
\begin{equation} \label{eq:systemphir2}
 \boldsymbol{\lambda}_j^\eps(\rho,r) = \mathbf{p}_j^\eps(\rho,r) + \sum_{m=1}^r \binom{r}{m} \jP^\eps(\rho,m) \mathsf{\Phi}_j^\eps(\rho,r-m).
\end{equation}

Relations \eqref{eq:systemphi0}, \eqref{eq:systemphir1}, and \eqref{eq:systemphir2} allows us to calculate mixed power-exponential moments of first hitting times for a fixed (sufficiently small) value of $\varepsilon$. In order to construct asymptotic expansions for these moments, we will use properties of asymptotic matrix expansions, which will be presented now.

Let $\mathbf{A}(\varepsilon)$ be an $m \times n$ matrix valued function. Suppose that $\mathbf{A}(\varepsilon)$ on the interval $0 < \varepsilon \leq \varepsilon_0$ can be represented as
\begin{equation*}
 \mathbf{A}(\varepsilon) = \mathbf{A}_0 + \mathbf{A}_1 \varepsilon + \cdots + \mathbf{A}_k \varepsilon^k + \mathbf{o}(\varepsilon^k),
\end{equation*}
where $\mathbf{A}_0,\ldots,\mathbf{A}_k$ are $m \times n$ matrices with real-valued elements and $\mathbf{o}(\varepsilon^k)$ is an $m \times n$ matrix where all elements are of order $o(\varepsilon^k)$. Then we say that $\mathbf{A}(\varepsilon)$ has an expansion of order $k$.

The following lemma collects some properties for asymptotic matrix expansions that will be used. These properties are known, but we give a short proof in order to make the paper more self-contained.

\begin{lemma} \label{lmm:matrixexp}
Let $\mathbf{A}(\varepsilon)$ be an $m \times n$ matrix valued function which has an expansion of order $k$, and let $\mathbf{B}(\varepsilon)$ be a $p \times q$ matrix valued function which has an expansion of order $l$.
\begin{enumerate}
 \item[$\mathbf{(i)}$] If $c$ is a real-valued constant, then $\mathbf{C}(\varepsilon) = c \mathbf{A}(\varepsilon)$ has an expansion of order $k$ and the coefficients are given by
 \begin{equation*}
  \mathbf{C}_i = c \mathbf{A}_i, \ i=0,1,\ldots,k.
 \end{equation*}
 \item[$\mathbf{(ii)}$] If $m = p$ and $n = q$, then $\mathbf{C}(\varepsilon) = \mathbf{A}(\varepsilon) + \mathbf{B}(\varepsilon)$ has an expansion of order $k \wedge l$ and the coefficients are given by
 \begin{equation*}
  \mathbf{C}_i = \mathbf{A}_i + \mathbf{B}_i, \ i=0,1,\ldots,k \wedge l.
 \end{equation*}
 \item[$\mathbf{(iii)}$] If $n = p$, then $\mathbf{C}(\varepsilon) = \mathbf{A}(\varepsilon) \mathbf{B}(\varepsilon)$ has an expansion of order $k \wedge l$ and the coefficients are given by
 \begin{equation*}
  \mathbf{C}_i = \sum_{j=0}^i \mathbf{A}_j \mathbf{B}_{i-j}, \ i=0,1,\ldots,k \wedge l.
 \end{equation*}
 \item[$\mathbf{(iv)}$] If $m = n$ and $\det( \mathbf{I} - \mathbf{A}_0 ) \neq 0$, then the inverse matrix $\mathbf{C}(\varepsilon) = (\mathbf{I} - \mathbf{A}(\varepsilon))^{-1}$ exists for sufficiently small $\varepsilon$ and has an expansion of order $k$ where the coefficients are given by
 \begin{equation*}
  \mathbf{C}_0 = (\mathbf{I} - \mathbf{A}_0)^{-1} \quad \text{and} \quad \mathbf{C}_i = \mathbf{C}_0 \sum_{j=1}^i \mathbf{A}_j \mathbf{C}_{i-j}, \ i=1,\ldots,k.
 \end{equation*}
\end{enumerate}
\end{lemma}

\begin{proof}
Parts $\mathbf{(i)}$, $\mathbf{(ii)}$, and $\mathbf{(iii)}$ are consequences of elementary algebraic relations.

For the proof of part $\mathbf{(iv)}$ we first note that since $(\mathbf{I} - \mathbf{A}(\varepsilon)) \rightarrow (\mathbf{I} - \mathbf{A}_0)$ as $\varepsilon \rightarrow 0$, and $\det( \mathbf{I} - \mathbf{A}_0 ) \neq 0$, it follows that $\det( \mathbf{I} - \mathbf{A}(\varepsilon) ) \neq 0$ for sufficiently small $\varepsilon$. Thus, the matrix $\mathbf{I} - \mathbf{A}(\varepsilon)$ has an inverse for sufficiently small $\varepsilon$. Furthermore, the elements of this inverse matrix are rational functions of the elements of $\mathbf{A}(\varepsilon)$. From this it follows that $( \mathbf{I} - \mathbf{A}(\varepsilon) )^{-1} \rightarrow ( \mathbf{I} - \mathbf{A}_0 )^{-1}$, so we have the representation
\begin{equation} \label{eq:expC0}
 \mathbf{C}(\varepsilon) = \mathbf{C}_0 + \mathbf{M}_0(\varepsilon),
\end{equation}
where $\mathbf{C}_0 = ( \mathbf{I} - \mathbf{A}_0 )^{-1}$ and $\mathbf{M}_0(\varepsilon) \rightarrow \mathbf{0}$ as $\varepsilon \rightarrow 0$.

Now assume that $k = 1$. Then, using \eqref{eq:expC0},
\begin{equation*}
\begin{split}
 \mathbf{I} &= ( \mathbf{I} - \mathbf{A}(\varepsilon) ) ( \mathbf{I} - \mathbf{A}(\varepsilon) )^{-1} \\
 &= ( \mathbf{I} - \mathbf{A}_0 - \mathbf{A}_1 \varepsilon + \mathbf{o}(\varepsilon) ) ( \mathbf{C}_0 + \mathbf{M}_0(\varepsilon) ) \\
 &= \mathbf{I} + ( \mathbf{I} - \mathbf{A}_0 ) \mathbf{M}_0(\varepsilon) - ( \mathbf{A}_1 \varepsilon + \mathbf{o}(\varepsilon) ) \mathbf{C}_0 + \mathbf{o}(\varepsilon).
\end{split}
\end{equation*}
Rewriting this relation and dividing by $\varepsilon > 0$, we get
\begin{equation*}
 \frac{\mathbf{M}_0(\varepsilon)}{\varepsilon} = (\mathbf{I} - \mathbf{A}_0)^{-1} \left( \mathbf{A}_1 + \frac{\mathbf{o}(\varepsilon)}{\varepsilon} \right) \mathbf{C}_0 + \frac{\mathbf{o}(\varepsilon)}{\varepsilon}.
\end{equation*}
Letting $\varepsilon$ tend to zero it follows that $\mathbf{M}_0(\varepsilon) / \varepsilon \rightarrow \mathbf{C}_0 \mathbf{A}_1 \mathbf{C}_0$ as $\varepsilon \rightarrow 0$. From this and relation \eqref{eq:expC0} we get the representation
\begin{equation*}
 \mathbf{C}(\varepsilon) = \mathbf{C}_0 + \mathbf{C}_1 \varepsilon + \mathbf{M}_1(\varepsilon),
\end{equation*}
where $\mathbf{C}_0 = ( \mathbf{I} - \mathbf{A}_0 )^{-1}$, $\mathbf{C}_1 = \mathbf{C}_0 \mathbf{A}_1 \mathbf{C}_0$ and $\mathbf{M}_1(\varepsilon) / \varepsilon \rightarrow \mathbf{0}$ as $\varepsilon \rightarrow 0$.

This proves part $\mathbf{(iv)}$ for $k = 1$.

For a general $k$ we can prove the result by induction using the same technique as above.
\end{proof}

We will now use the results above to show how mixed power-exponential moments of first hitting times can be expanded in a power series with respect to the perturbation parameter and how the coefficients can be calculated explicitly.

Let us introduce the following perturbation condition which is assumed to hold for some $\rho < \delta$, where $\delta$ is the number from Lemma \ref{lmm:mgfproperties}:
\begin{enumerate}
 \item[$\Cperturbationmom$:] $p_{i j}^\eps(\rho,r) = p_{i j}^\zero(\rho,r) + p_{i j}[\rho,r,1] \varepsilon + \cdots + p_{i j}[\rho,r,k-r] \varepsilon^{k-r} + o(\varepsilon^{k-r})$, $r=0,\ldots,k$, $i,j \neq 0$, where $|p_{i j}[\rho,r,n]| < \infty$, $r=0,\ldots,k$, $n=1,\ldots,k-r$, $i,j \neq 0$.
\end{enumerate}
For convenience we denote $p_{i j}[\rho,r,0] = p_{i j}^\zero(\rho,r)$, for $r = 0,\ldots,k$.

To prepare for the next result, note that it follows from condition $\Cperturbationmom$ that the vectors $\mathbf{p}_j^\eps(\rho,r)$ and matrices $\jP^\eps(\rho,r)$, defined by relations \eqref{eq:vectorp} and \eqref{eq:matrixjP}, respectively, have asymptotic expansions
\begin{equation*}
 \mathbf{p}_j^\eps(\rho,r) = \mathbf{p}_j^\zero(\rho,r) + \mathbf{p}_j[\rho,r,1] \varepsilon + \cdots + \mathbf{p}_j[\rho,r,k-r] \varepsilon^{k-r} + \mathbf{o}(\varepsilon^{k-r}),
\end{equation*}
and
\begin{equation*}
 \jP^\eps(\rho,r) = \jP^\zero(\rho,r) + \jP[\rho,r,1] \varepsilon + \cdots + \jP[\rho,r,k-r] \varepsilon^{k-r} + \mathbf{o}(\varepsilon^{k-r}),
\end{equation*}
where the vector coefficients $\mathbf{p}_j[\rho,r,n]$ are given by
\begin{equation*}
 \mathbf{p}_j[\rho,r,n] = \begin{bmatrix} p_{1 j}[\rho,r,n] & p_{2 j}[\rho,r,n] & \cdots & p_{N j}[\rho,r,n] \end{bmatrix}^T,
\end{equation*}
and the coefficients $\jP[\rho,r,n] = \| \jp_{i k}[\rho,r,n] \|$ are $N \times N$ matrices where the elements are given by
\begin{equation*}
 \jp_{i k}[\rho,r,n] = \left\{
 \begin{array}{l l}
  p_{i k}[\rho,r,n] & i=1,\ldots,N, \ k \neq j, \\
  0 & i=1,\ldots,N, \ k = j.
 \end{array}
 \right.
\end{equation*}

The following theorem is an essential tool for the proof of the main result of the present paper.

\begin{theorem} \label{thm:mixedmoments}
Assume that conditions $\Ccontinuity$, $\Cergodicity$, $\Cmgf$, and $\Cperturbationmom$ hold and fix some $j \neq 0$. Then:
\begin{enumerate}
 \item[$\mathbf{(i)}$] The inverse matrix $\jU^\eps(\rho) = (\mathbf{I} - \jP^\eps(\rho))^{-1}$ exists for sufficiently small $\varepsilon$ and has the expansion
 \begin{equation*}
  \jU^\eps(\rho) = \jU[\rho,0] + \jU[\rho,1] \varepsilon + \cdots + \jU[\rho,k] + \mathbf{o}(\varepsilon^k),
 \end{equation*}
 where
 \begin{equation*}
  \jU[\rho,n] = \left\{
  \begin{array}{l l}
   (\mathbf{I} - \jP^\zero(\rho))^{-1} & n=0, \\
   \jU[\rho,0] \sum_{q=1}^n \jP[\rho,0,q] \jU[\rho,n-q] & n=1,\ldots,k.
  \end{array}
  \right.
 \end{equation*}
 \item[$\mathbf{(ii)}$] We have the expansion
 \begin{equation*}
  \mathsf{\Phi}_j^\eps(\rho) = \mathsf{\Phi}_j[\rho,0,0] + \mathsf{\Phi}_j[\rho,0,1] \varepsilon + \cdots + \mathsf{\Phi}_j[\rho,0,k] \varepsilon^k + \mathbf{o}(\varepsilon^k),
 \end{equation*}
 where
 \begin{equation} \label{eq:phi0coeff}
  \mathsf{\Phi}_j[\rho,0,n] = \left\{
  \begin{array}{l l}
   \mathsf{\Phi}_j^\zero(\rho) & n=0, \\
   \sum_{q=0}^n \jU[\rho,q] \mathbf{p}_j[\rho,0,n-q] & n=1,\ldots,k.
  \end{array}
  \right.
 \end{equation}
 \item[$\mathbf{(iii)}$] For $r=1,\ldots,k,$ we have the expansion
 \begin{equation*}
  \mathsf{\Phi}_j^\eps(\rho,r) = \mathsf{\Phi}_j[\rho,r,0] + \mathsf{\Phi}_j[\rho,r,1] \varepsilon + \cdots + \mathsf{\Phi}_j[\rho,r,k-r] \varepsilon^{k-r} + \mathbf{o}(\varepsilon^{k-r}),
 \end{equation*}
 where the coefficients can be calculated recursively by the formulas
 \begin{equation*}
  \mathsf{\Phi}_j[\rho,r,n] = \left\{
  \begin{array}{l l}
   \mathsf{\Phi}_j^\zero(\rho,r) & n=0, \\
   \sum_{q=0}^n \jU[\rho,q] \boldsymbol{\lambda}_j[\rho,r,n-q], & n=1,\ldots,k-r,
  \end{array}
  \right.
 \end{equation*}
 where, for $s=0,\ldots,k-r$,
 \begin{equation*}
  \boldsymbol{\lambda}_j[\rho,r,s] = \mathbf{p}_j[\rho,r,s] + \sum_{m=1}^r \binom{r}{m} \sum_{q=0}^s \jP[\rho,m,q] \mathsf{\Phi}_j[\rho,r-m,s-q].
 \end{equation*}
\end{enumerate}
\end{theorem}

\begin{proof}
First note that under conditions $\Ccontinuity$, $\Cergodicity$, and $\Cmgf$, it follows from part $\mathbf{(iii)}$ of Lemma \ref{lmm:mgfproperties} that $\mathsf{\Phi}_j^\zero(\rho) < \infty$, for all $\rho \leq \delta$. Thus, by applying Lemma \ref{lmm:finitemgfs} we see that the inverse matrix $(\mathbf{I} - \jP^\zero(\rho))^{-1}$ exists for all $\rho \leq \delta$. Using this and condition $\Cperturbationmom$, part $\mathbf{(i)}$ now follows from part $\mathbf{(iv)}$ of Lemma \ref{lmm:matrixexp}.

For the proof of part $\mathbf{(ii)}$ notice that it follows from Equation \eqref{eq:systemphi0} and part $\mathbf{(i)}$ that for sufficiently small $\varepsilon$ we have
\begin{equation} \label{eq:phi0solution}
 \mathsf{\Phi}_j^\eps(\rho) = (\mathbf{I} - \jP^\eps(\rho))^{-1} \mathbf{p}_j^\eps(\rho).
\end{equation}
It follows from \eqref{eq:phi0solution}, part $\mathbf{(i)}$, condition $\Cperturbationmom$, and part $\mathbf{(iii)}$ of Lemma \ref{lmm:matrixexp} that $\mathsf{\Phi}_j^\eps(\rho)$ has an expansion of order $k$ with coefficients given by Equation \eqref{eq:phi0coeff}. This proves part $\mathbf{(ii)}$.

Now we consider Equations \eqref{eq:systemphir1} and \eqref{eq:systemphir2} for $r = 1$:
\begin{equation} \label{eq:phi1expr}
 \mathsf{\Phi}_j^\eps(\rho,1) = \boldsymbol{\lambda}_j^\eps(\rho,1) + \jP^\eps(\rho) \mathsf{\Phi}_j^\eps(\rho,1),
\end{equation}
where
\begin{equation} \label{eq:lambda1expr}
 \boldsymbol{\lambda}_j^\eps(\rho,1) = \mathbf{p}_j^\eps(\rho,1) + \jP^\eps(\rho,1) \mathsf{\Phi}_j^\eps(\rho).
\end{equation}
It follows from \eqref{eq:phi1expr} and part $\mathbf{(i)}$ that for sufficiently small $\varepsilon$,
\begin{equation} \label{eq:phi1solution}
 \mathsf{\Phi}_j^\eps(\rho,1) = (\mathbf{I} - \jP^\eps(\rho))^{-1} \boldsymbol{\lambda}_j^\eps(\rho,1).
\end{equation}
It follows from \eqref{eq:lambda1expr}, part $\mathbf{(ii)}$, condition $\Cperturbationmom$, and parts $\mathbf{(ii)}$-$\mathbf{(iii)}$ of Lemma \ref{lmm:matrixexp} that
\begin{equation} \label{eq:lambda1exp}
 \boldsymbol{\lambda}_j^\eps(\rho,1) = \boldsymbol{\lambda}_j[\rho,1,0] + \boldsymbol{\lambda}_j[\rho,1,1] \varepsilon + \cdots + \boldsymbol{\lambda}_j[\rho,1,k-1] \varepsilon^{k-1} + \mathbf{o}(\varepsilon^{k-1}),
\end{equation}
where
\begin{equation} \label{eq:lambda1coeff}
 \boldsymbol{\lambda}_j[\rho,1,s] = \mathbf{p}_j[\rho,1,s] + \sum_{q=0}^s \jP[\rho,1,q] \mathsf{\Phi}_j[\rho,0,s-q], \ s=0,\ldots,k-1.
\end{equation}
It now follows from \eqref{eq:phi1solution}, \eqref{eq:lambda1exp}, \eqref{eq:lambda1coeff}, part $\mathbf{(i)}$, and part $\mathbf{(iii)}$ of Lemma \ref{lmm:matrixexp} that $\mathsf{\Phi}_j^\eps(\rho,1)$ has an expansion of order $k-1$ with coefficients given by
\begin{equation*}
 \mathsf{\Phi}_j[\rho,1,n] = \sum_{q=0}^n \jU[\rho,q] \boldsymbol{\lambda}_j[\rho,1,n-q], \ n=1,\ldots,k-1.
\end{equation*}
This proves part $\mathbf{(iii)}$ for $r = 1$.

We prove the general result by induction. Let us assume that part $\mathbf{(iii)}$ holds for $r=1,\ldots,u-1$, for some $u \leq k$. Equations \eqref{eq:systemphir1} and \eqref{eq:systemphir2} give
\begin{equation} \label{eq:phikexpr}
 \mathsf{\Phi}_j^\eps(\rho,u) = \boldsymbol{\lambda}_j^\eps(\rho,u) + \jP^\eps(\rho) \mathsf{\Phi}_j^\eps(\rho,u),
\end{equation}
where
\begin{equation} \label{eq:lambdakexpr}
 \boldsymbol{\lambda}_j^\eps(\rho,u) = \mathbf{p}_j^\eps(\rho,u) + \sum_{m=1}^u \binom{u}{m} \jP^\eps(\rho,m) \mathsf{\Phi}_j^\eps(\rho,u-m).
\end{equation}
It follows from \eqref{eq:phikexpr} and part $\mathbf{(i)}$ that for sufficiently small $\varepsilon$,
\begin{equation} \label{eq:phiksolution}
 \mathsf{\Phi}_j^\eps(\rho,u) = (\mathbf{I} - \jP^\eps(\rho))^{-1} \boldsymbol{\lambda}_j^\eps(\rho,u).
\end{equation}
It follows from \eqref{eq:lambdakexpr}, part $\mathbf{(ii)}$, condition $\Cperturbationmom$, parts $\mathbf{(i)}$-$\mathbf{(iii)}$ of Lemma \ref{lmm:matrixexp}, and the induction hypothesis that
\begin{equation} \label{eq:lambdakexp}
 \boldsymbol{\lambda}_j^\eps(\rho,u) = \boldsymbol{\lambda}_j[\rho,u,0] + \boldsymbol{\lambda}_j[\rho,u,1] \varepsilon + \cdots + \boldsymbol{\lambda}_j[\rho,u,k-u] \varepsilon^{k-u} + \mathbf{o}(\varepsilon^{k-u}),
\end{equation}
where for $s=0,\ldots,k-u$,
\begin{equation} \label{eq:lambdakcoeff}
 \boldsymbol{\lambda}_j[\rho,u,s] = \mathbf{p}_j[\rho,u,s] + \sum_{m=0}^u \binom{u}{m} \sum_{q=0}^s \jP[\rho,m,q] \mathsf{\Phi}_j[\rho,u-m,s-q].
\end{equation}
It now follows from \eqref{eq:phiksolution}, \eqref{eq:lambdakexp}, \eqref{eq:lambdakcoeff}, part $\mathbf{(i)}$, and part $\mathbf{(iii)}$ of Lemma \ref{lmm:matrixexp} that $\mathsf{\Phi}_j^\eps(\rho,u)$ has an expansion of order $k-u$ with coefficients given by
\begin{equation*}
 \mathsf{\Phi}_j[\rho,u,n] = \sum_{q=0}^n \jU[\rho,q] \boldsymbol{\lambda}_j[\rho,u,n-q], \ n=1,\ldots,k-u.
\end{equation*}
This concludes the proof of Theorem \ref{thm:mixedmoments}.
\end{proof}

\section{Solidarity Property of Periodicity} \label{sec:periodicity}

In this section we show that the periodicity of the distribution of first return time satisfies a solidarity property.

The period of $g_{i i}^\eps(n)$ is defined by
\begin{equation*}
 d_i = \gcd \{ n \in \mathbb{N} : g_{i i}^\eps(n) > 0 \}, \ i \neq 0.
\end{equation*}
In particular, $d_i = 1$ means that $g_{i i}^\eps(n)$ is non-periodic.

In order to guarantee non-periodicity of $g_{i i}^\zero(n)$, we will assume that the following condition holds:
\begin{enumerate}
 \item[$\Cnonperiodicity$:] $g_{j j}^\zero(n)$ is non-periodic for some $j \neq 0$.
\end{enumerate}

It will be shown that the function $g_{i i}^\eps(n)$ have the same period for all states $i \neq 0$. In the proof of this result we will use the convolution operator. For two real-valued functions $f(n)$, $n=0,1,\ldots,$ and $g(n)$, $n=0,1,\ldots,$ the convolution is defined by
\begin{equation*}
 f * g(n) = \sum_{k=0}^n f(n-k) g(k), \ n=0,1,\ldots
\end{equation*}
Furthermore, for a function $f(n)$, $n=0,1,\ldots,$ the $k$-fold convolution $f^{(*k)}(n)$ is defined recursively by $f^{(*0)}(n) = \chi(n=0)$ and
\begin{equation*}
 f^{(*k)}(n) = f * f^{(*(k-1))}(n), \ k=1,2,\ldots
\end{equation*}
Notice that $f^{(*1)}(n) = f(n)$.

Let us introduce the following notation:
\begin{equation*}
 \kg_{i j}^\eps(n) = \mathsf{P}_i \{ \mu_j = n, \ \nu_0^\eps \wedge \nu_k^\eps > \nu_j^\eps \}, \ n=0,1,\ldots, \ i,j,k \in X.
\end{equation*}

In the proof of the following lemma we adopt a technique that is used in the proof of a similar result for continuous time semi-Markov processes given in {\c C}inlar (1974).

\begin{lemma} \label{lmm:solperiod}
If we for some $\varepsilon \geq 0$ have $g_{i j}^\eps > 0$ for all $i,j \neq 0$, then $d_i = d_j$ for all $i,j \neq 0$.
\end{lemma}

\begin{proof}
Choose $i,j \neq 0$ arbitrarily. The conclusion is trivial if $i = j$ so let us assume that $i \neq j$.

By using the regenerative property of the semi-Markov process we can for all $n=0,1,\ldots,$ write down the following relations:
\begin{equation} \label{eq:solperiod1}
 g_{i i}^\eps(n) = \jg_{i i}^\eps(n) + \ig_{i j}^\eps * g_{j i}^\eps(n),
\end{equation}
\begin{equation} \label{eq:solperiod2}
 g_{j i}^\eps(n) = \jg_{j i}^\eps(n) + \ig_{j j}^\eps * g_{j i}^\eps(n),
\end{equation}
\begin{equation} \label{eq:solperiod3}
 g_{j j}^\eps(n) = \ig_{j j}^\eps(n) + \jg_{j i}^\eps * g_{i j}^\eps(n),
\end{equation}
\begin{equation} \label{eq:solperiod4}
 g_{i j}^\eps(n) = \ig_{i j}^\eps(n) + \jg_{i i}^\eps * g_{i j}^\eps(n).
\end{equation}
Iterating Equation \eqref{eq:solperiod2} and using \eqref{eq:solperiod1} we get
\begin{equation} \label{eq:solperiod5}
\begin{split}
 g_{i i}^\eps(n) &= \jg_{i i}^\eps(n) + \sum_{k=0}^m \ig_{i j}^\eps * (\ig_{j j}^\eps)^{(*k)} * \jg_{j i}^\eps (n) \\
 &\quad \quad + \ig_{i j}^\eps * (\ig_{j j}^\eps)^{(*(m+1))} * g_{j i}^\eps (n), \ m=0,1,\ldots
\end{split}
\end{equation}
Similarly, by using Equations \eqref{eq:solperiod3} and \eqref{eq:solperiod4} we get
\begin{equation} \label{eq:solperiod6}
\begin{split}
 g_{j j}^\eps(n) &= \ig_{j j}^\eps(n) + \sum_{k=0}^m \jg_{j i}^\eps * (\jg_{i i}^\eps)^{(*k)} * \ig_{i j}^\eps (n) \\
 &\quad \quad + \jg_{j i}^\eps * (\jg_{i i}^\eps)^{(*(m+1))} * g_{i j}^\eps (n), \ m=0,1,\ldots
\end{split}
\end{equation}

Since $g_{i i}^\eps(n)$ has period $d_i$, it has all its mass concentrated on the set $d_i \mathbb{N} = \{d_i,2 d_i,\ldots \}$. It follows from \eqref{eq:solperiod5} with $m = 0$ that the functions $\jg_{i i}^\eps(n)$, $\ig_{i j}^\eps * \jg_{j i}^\eps (n)$ and $\ig_{i j}^\eps * \ig_{j j}^\eps * \jg_{j i}^\eps (n)$ are all concentrated on the set $d_i \mathbb{N}$. Since $\ig_{i j}^\eps * \jg_{j i}^\eps (n)$ is not identically equal to zero, it also follows from \eqref{eq:solperiod5} that $\ig_{j j}^\eps(n)$ concentrates on $d_i \mathbb{N}$. It can now be concluded that all functions on the right hand side of \eqref{eq:solperiod6}, except for possibly the last one, is concentrated on $d_i \mathbb{N}$. Using this, and that $g_{j j}^\eps(n)$ is the limit of the right hand side of \eqref{eq:solperiod6} as $m \rightarrow \infty$, we have for any $n' \notin d_i \mathbb{N}$,
\begin{equation*}
 g_{j j}^\eps(n') = \lim_{m \rightarrow \infty} \jg_{j i}^\eps * (\jg_{i i}^\eps)^{(*(m+1))} * g_{i j}^\eps (n') = 0.
\end{equation*}
This means that $g_{j j}^\eps(n)$ is concentrated on the set $d_i \mathbb{N}$ and we can conclude that $d_j \geq d_i$. By using analogous arguments as above, \eqref{eq:solperiod5} and \eqref{eq:solperiod6} can also be used to show that $d_i \geq d_j$. In conclusion, $d_i = d_j$.
\end{proof}

\section{Exponential Expansions for Perturbed \\ Semi-Markov Processes} \label{sec:mainresult}

In this section we give asymptotic exponential expansions for perturbed discrete time semi-Markov processes with absorption. The results are obtained by applying corresponding results for perturbed regenerative processes given in Section \ref{sec:regenerative}.

Our main objective is to give a detailed asymptotic analysis of the probabilities
\begin{equation*}
 P_{i j}^\eps(n) = \mathsf{P}_i \{ \xi^\eps(n) = j, \ \mu_0^\eps > n \}, \ n=0,1,\ldots, \ i,j \neq 0,
\end{equation*}
as $n \rightarrow \infty$ and $\varepsilon \rightarrow 0$.

Let us assume that the initial distribution of the semi-Markov process $\xi^\eps(n)$ is concentrated at some state $i \neq 0$. Then $\xi^\eps(n)$ is a regenerative process with regeneration times being successive return times to state $i$. If state $0$ is an absorbing state, these regeneration times are possibly improper random variables. In Section \ref{sec:regenerative} it was assumed that the regeneration times were proper random variables. However, the probabilities $P_{i j}^\eps(n)$, $i,j \neq 0$, do not depend on the transition probabilities from state $0$. This means that we can modify these transition probabilities without affecting the probabilities $P_{i j}^\eps(n)$, $i,j \neq 0$. For example, if we take $Q_{i j}^\eps(n) = \chi( n = 1 )/(N+1)$, then return times to any fixed initial state $i \neq 0$ can serve as proper regeneration times. We can apply the results of Section \ref{sec:regenerative} to this modified process and then it follows that the results also hold for the process where $0$ is an absorbing state.

By using the regenerative property of the semi-Markov process at return times to the initial state, we can for any $i,j \neq 0$ write the following renewal equation:
\begin{equation*}
 P_{i j}^\eps(n) = h_{i j}^\eps(n) + \sum_{k=0}^n P_{i j}^\eps(n-k) g_{i i}^\eps(k), \ n=0,1,\ldots,
\end{equation*}
where
\begin{equation*}
 h_{i j}^\eps(n) = \mathsf{P}_i \{ \xi^\eps(n) = j, \ \mu_0^\eps \wedge \mu_i^\eps > n \}.
\end{equation*}
It follows that $\mu_0^\eps$, the first hitting time of state $0$, is a regenerative stopping time for $\xi^\eps(n)$.

For the model of perturbed semi-Markov processes, the characteristic equation takes the form
\begin{equation} \label{eq:chareq}
 \phi_{i i}^\eps(\rho) = 1.
\end{equation}
It will be shown that Equation \eqref{eq:chareq} has a unique solution $\rho^\eps$ for sufficiently small $\varepsilon$ that does not depend on $i$.

Furthermore, let us define
\begin{equation*}
 \widetilde{\pi}_{i j}^\zero = \frac{\sum_{n=0}^\infty e^{\rho^\zero n} h_{i j}^\zero(n)}{\sum_{n=0}^\infty n e^{\rho^\zero n} g_{i i}^\zero(n)}, \ i,j \neq 0.
\end{equation*}

It is interesting to note that in the pseudo-stationary case, $\widetilde{\pi}_{i j}^\zero$ does not depend on $i$. Indeed, in this case $\rho^\zero = 0$ and $\mu_0^\zero = \infty$ almost surely, so we get
\begin{equation*}
 \widetilde{\pi}_{i j}^\zero = \frac{\sum_{n=0}^\infty h_{i j}^\zero(n)}{\sum_{n=0}^\infty n g_{i i}^\zero(n)} = \frac{\mathsf{E}_i \sum_{n=0}^\infty \chi( \xi^\zero(n) = j, \ \mu_i^\zero > n )}{\mathsf{E}_i \mu_i^\zero}, \ i,j \neq 0.
\end{equation*}
That is, $\widetilde{\pi}_{i j}^\zero$ is the quotient of the expected number of visits to state $j$ during an excursion starting from state $i$ and the expected length of this excursion for the limiting semi-Markov process. It is known that this quantity does not depend on state $i$. Moreover, in this case $\pi_j^\zero = \widetilde{\pi}_{i j}^\zero$, $j=1,\ldots,N$, are the stationary probabilities for the limiting semi-Markov process.

Let us formulate condition $\Cperturbationmom$ for $\rho = \rho^\zero$:
\begin{enumerate}
 \item[$\Cperturbation$:] $p_{i j}^\eps(\rho^\zero,r) = p_{i j}^\zero(\rho^\zero,r) + p_{i j}[\rho^\zero,r,1] \varepsilon + \cdots + p_{i j}[\rho^\zero,r,k-r] \varepsilon^{k-r} + o(\varepsilon^{k-r})$, $r=0,\ldots,k$, $i,j \neq 0$, where $|p_{i j}[\rho^\zero,r,n]| < \infty$, $r=0,\ldots,k$, $n=1,\ldots,k-r$, $i,j \neq 0$.
\end{enumerate}

Under conditions $\Ccontinuity$--$\Cperturbation$ it follows from Theorem \ref{thm:mixedmoments} that we for each $i \neq 0$ and $r=0,\ldots,k$ have the asymptotic expansion
\begin{equation*}
 \phi_{i i}^\eps(\rho^\zero,r) = b_i[r,0] + b_i[r,1] \varepsilon + \cdots + b_i[r,k-r] \varepsilon^{k-r} + o(\varepsilon^{k-r}),
\end{equation*}
where $b_i[r,0] = \phi_{i i}^\zero(\rho^\zero,r)$, $r=0,\ldots,k$, $i \neq 0$, and the coefficients $b_i[r,n]$, $r=0,\ldots,k$, $n=1,\ldots,k-r$, $i \neq 0$, can be calculated from the recursive formulas given in this theorem.

We now present the main result of this paper.

\begin{theorem} \label{thm:expexp}
Assume that conditions $\Ccontinuity$--$\Cnonperiodicity$ hold. Then:
\begin{enumerate}
\item[$\mathbf{(i)}$] For $\varepsilon$ sufficiently small, there exists a unique root $\rho^\eps$ of the characteristic equation \eqref{eq:chareq} which does not depend on the choice of initial state $i$. Moreover, we have the asymptotic expansion
\begin{equation*}
 \rho^\eps = \rho^\zero + c_1 \varepsilon + \cdots + c_k \varepsilon^k + o(\varepsilon^k),
\end{equation*}
where $c_1 = - b_i[0,1] / b_i[1,0]$ and for $n = 2,\ldots,k$,
\begin{equation*}
\begin{split}
 c_n = - \frac{1}{b_i[1,0]} &\Bigg( b_i[0,n] + \sum_{q=1}^{n-1} b_i[1,n-q] c_q \\
 &+ \sum_{m=2}^n \sum_{q=m}^n b_i[m,n-q] \cdot \sum_{n_1,\ldots,n_{q-1} \in D_{m,q}} \prod_{p=1}^{q-1} \frac{c_p^{n_p}}{n_p!} \Bigg),
\end{split}
\end{equation*}
with $D_{m,q}$ being the set of all non-negative integer solutions to the system
\begin{equation*}
 n_1 + \cdots + n_{q-1} = m, \quad n_1 + \cdots + (q-1) n_{q-1} = q.
\end{equation*}
\item[$\mathbf{(ii)}$] For any non-negative integer valued function $n^\eps \rightarrow \infty$ as $\varepsilon \rightarrow 0$ in such a way that $\varepsilon^r n^\eps \rightarrow \lambda_r \in [0,\infty)$ for some $1 \leq r \leq k$, we have
\begin{equation*}
 \frac{ \mathsf{P}_i \{ \xi^\eps(n^\eps) = j, \ \mu_0^\eps > n^\eps \} }{\exp( -(\rho^\zero + c_1 \varepsilon + \cdots + c_{r-1} \varepsilon^{r-1}) n^\eps )} \rightarrow \frac{\widetilde{\pi}_{i j}^\zero}{e^{\lambda_r c_r}} \ \text{as} \ \varepsilon \rightarrow 0, \ i,j \neq 0.
\end{equation*}
\end{enumerate}
\end{theorem}

\begin{proof}
Throughout the proof, we let the initial state $i \neq 0$ be fixed. It will be shown that conditions $\Ccontinuity$--$\Cnonperiodicity$ imply that conditions $\Ccontinuityreg$--$\Ctailprobreg$ hold for the functions
\begin{equation*}
 f^\eps(n) = g_{i i}^\eps(n), \ n=0,1,\ldots,
\end{equation*}
and
\begin{equation*}
 q^\eps(n,A) = \sum_{j \in A} h_{i j}^\eps(n), \ n=0,1,\ldots, \ A \subseteq X.
\end{equation*}
Then, Theorem \ref{thm:expexpreg} can be applied in order to prove Theorem \ref{thm:expexp}.

Let us first show that the function
\begin{equation*}
 f(n) = g_{i i}^\eps(n) = \mathsf{P}_i \{ \mu_i^\eps = n, \ \nu_0^\eps > \nu_i^\eps \}, \ n=0,1,\ldots
\end{equation*}
satisfies condition $\Ccontinuityreg$.

As was shown in Section \ref{sec:moments}, the vector of moment generating functions $\mathsf{\Phi}_i^\eps(\rho)$ satisfies the following system of linear equations:
\begin{equation} \label{eq:expexp1}
 \mathsf{\Phi}_i^\eps(\rho) = \mathbf{p}_i^\eps(\rho) + \iP^\eps(\rho) \mathsf{\Phi}_i^\eps(\rho).
\end{equation}
It follows from part $\mathbf{(iv)}$ of Lemma \ref{lmm:mgfproperties} that there exist $\varepsilon_1 > 0$ and $\delta > 0$ such that $\mathsf{\Phi}_i^\eps(\rho) < \infty$ for all $\varepsilon \leq \varepsilon_1$ and $\rho \leq \delta$. Thus, we can use Lemma \ref{lmm:finitemgfs} to conclude that the system \eqref{eq:expexp1} has a unique solution for $\varepsilon \leq \varepsilon_1$ and $\rho \leq \delta$ given by
\begin{equation} \label{eq:expexp2}
 \mathsf{\Phi}_i^\eps(\rho) = ( \mathbf{I} - \iP^\eps(\rho) )^{-1} \mathbf{p}_i^\eps(\rho).
\end{equation}
Using \eqref{eq:expexp2} and condition $\Ccontinuity$ it follows that $\mathsf{\Phi}_i^\eps(\rho) \rightarrow \mathsf{\Phi}_i^\zero(\rho)$ as $\varepsilon \rightarrow 0$ for $\rho \leq \delta$ and in particular
\begin{equation} \label{eq:expexp3}
 \phi_{i i}^\eps(\rho) \rightarrow \phi_{i i}^\zero(\rho) \ \text{as} \ \varepsilon \rightarrow 0, \ \rho \leq \delta.
\end{equation}

Relation \eqref{eq:expexp3} implies that for all $n=0,1,\ldots,$ we have $g_{i i}^\eps(n) \rightarrow g_{i i}^\zero(n)$ as $\varepsilon \rightarrow 0$. Since $\phi_{i i}^\eps(0) = g_{i i}^\eps$, relation \eqref{eq:expexp3} also implies that $g_{i i}^\eps \rightarrow g_{i i}^\zero$ as $\varepsilon \rightarrow 0$. Furthermore, by condition $\Cergodicity$, the function $g_{i i}^\zero(n)$ is not concentrated at zero and by applying Lemma \ref{lmm:solperiod} under condition $\Cnonperiodicity$, we see that $g_{i i}^\zero(n)$ is non-periodic. Thus, the function $g_{i i}^\eps(n)$ satisfies condition $\Ccontinuityreg$.

It follows from Lemma \ref{lmm:mgfproperties} that the moment generating function
\begin{equation*}
 \phi^\eps(\rho) = \phi_{i i}^\eps(\rho) = \sum_{n=0}^\infty e^{\rho n} g_{i i}^\eps(n), \ \rho \in \mathbb{R},
\end{equation*}
satisfies condition $\Cmgfreg$.

Applying Lemma \ref{lmm:chareqreg} now shows that there exists a unique non-negative solution $\rho_i^\eps$ of the characteristic equation $\phi_{i i}^\eps(\rho) = 1$ for sufficiently small $\varepsilon$, say $\varepsilon \leq \varepsilon_2$. Now for any $j \neq i$ and $\varepsilon \leq \varepsilon_2$ we can apply the same arguments as in the proof of part $\mathbf{(i)}$ of Lemma \ref{lmm:mgfproperties} to see that we also have $\phi_{j j}^\eps(\rho_i^\eps) = 1$. Thus, it can be concluded that the root of the characteristic equation \eqref{eq:chareq} does not depend on the initial state $i$ and we can drop the index and just write $\rho^\eps$.

It follows from Theorem \ref{thm:mixedmoments} that condition $\Cperturbationreg$ holds for the moments
\begin{equation*}
 \phi^\eps(\rho^\zero,r) = \phi_{i i}^\eps(\rho^\zero,r), \ r=0,\ldots,k.
\end{equation*}

Part $\mathbf{(i)}$ now follows by applying part $\mathbf{(i)}$ of Theorem \ref{thm:expexpreg}.

To prove part $\mathbf{(ii)}$ we also need to show that the function
\begin{equation*}
 q^\eps(n,X) = \sum_{j \in X} h_{i j}^\eps(n) = \mathsf{P}_i \{ \mu_0^\eps \wedge \mu_i^\eps > n \}, \ n=0,1,\ldots,
\end{equation*}
satisfies condition $\Ctailprobreg$. Thus, we need to show that there exists $\gamma > 0$ such that
\begin{equation} \label{eq:expexp4}
 \limsup_{0 \leq \varepsilon \rightarrow 0} \sum_{n=0}^\infty e^{(\rho^\zero + \gamma) n} \mathsf{P}_i \{ \mu_0^\eps \wedge \mu_i^\eps > n \} < \infty.
\end{equation}

In order to do this, first note that for any $\rho \neq 0$ we have
\begin{equation} \label{eq:expexp5}
\begin{split}
 \sum_{n=0}^\infty e^{\rho n} \mathsf{P}_i \{ \mu_0^\eps \wedge \mu_i^\eps > n \} &= \sum_{n=0}^\infty \sum_{k=n+1}^\infty e^{\rho n} \mathsf{P}_i \{ \mu_0^\eps \wedge \mu_i^\eps = k \} \\
 &= \sum_{k=1}^\infty \frac{e^{\rho k} - 1}{e^\rho - 1} \mathsf{P}_i \{ \mu_0^\eps \wedge \mu_i^\eps = k \} \\
 &= \frac{\mathsf{E}_i e^{\rho (\mu_0^\eps \wedge \mu_i^\eps)} - 1}{e^\rho - 1}.
\end{split}
\end{equation}

By Lemma \ref{lmm:mgfproperties} there exist $\delta \in (0,\beta]$ and $\varepsilon_3 > 0$ such that $\mathsf{\Phi}_i^\eps(\delta) < \infty$, for all $\varepsilon \leq \varepsilon_3$. From this, Lemma \ref{lmm:finitemgfs} implies that for any $\varepsilon \leq \varepsilon_3$, we have $\iP^\eps(\delta) < \infty$ and the inverse matrix $( \mathbf{I} - \iP^\eps(\delta) )^{-1}$ exists. Moreover, since $\delta \leq \beta$, condition $\Cmgf$ gives that there exists $\varepsilon_4 > 0$ such that $\mathbf{p}_0^\eps(\delta) < \infty$ for $\varepsilon \leq \varepsilon_4$. By Lemma \ref{lmm:finitemgfs2}, it can now be concluded that $\widetilde{\mathsf{\Phi}}_i^\eps(\delta) < \infty$ for $\varepsilon \leq \min \{ \varepsilon_3, \varepsilon_4 \}$. Using this we get
\begin{equation} \label{eq:expexp6}
 \mathsf{E}_i e^{\delta (\mu_0^\eps \wedge \mu_i^\eps)} = \phi_{i i}^\eps(\delta) + \widetilde{\phi}_{i i}^\eps(\delta) < \infty, \ \varepsilon \leq \min \{ \varepsilon_3, \varepsilon_4 \}.
\end{equation}

It follows from Lemma \ref{lmm:chareqreg} that $\rho^\zero < \delta$, so there exists $\gamma > 0$ such that
\begin{equation} \label{eq:expexp7}
 \rho^\zero + \gamma < \delta.
\end{equation}

Relation \eqref{eq:expexp4} now follows from \eqref{eq:expexp5}, \eqref{eq:expexp6}, and \eqref{eq:expexp7}.

Applying part $\mathbf{(ii)}$ of Theorem \ref{thm:expexpreg} now shows that part $\mathbf{(ii)}$ of Theorem \ref{thm:expexp} holds for all $j \neq 0$ for which we have
\begin{equation} \label{eq:expexp8}
 h_{i j}^\eps(n) \rightarrow h_{i j}^\zero(n) \ \text{as} \ \varepsilon \rightarrow 0, \ n=0,1,\ldots
\end{equation}
However, under condition $\Ccontinuity$, relation \eqref{eq:expexp8} holds for all $j \neq 0$ since it is possible to write $h_{i j}^\eps(n)$ as a finite sum where each term in the sum is a continuous function of quantities given in condition $\Ccontinuity$. This concludes the proof of Theorem \ref{thm:expexp}.
\end{proof}

% Bibliography

\end{document}